\documentclass[12pt]{article}

\usepackage{amssymb,a4}
\usepackage{amsmath,amsfonts,amssymb,amsthm}
\usepackage{bm}

\newcommand{\Z}{{\mathbb Z}}

\newcommand{\C}{{\mathbb C}}

\newcommand{\N}{{\mathbb N}}
\newcommand{\U}{{\mathcal U}}
\newcommand{\wtL}{{\widetilde{L}}}
\newcommand{\wtM}{{\widetilde{M}}}

\def\<{\langle}
\def\>{\rangle}

\newtheorem{thm}{Theorem}[section]
\newtheorem{prop}[thm]{Proposition}
\newtheorem{lem}[thm]{Lemma}
\newtheorem{cor}[thm]{Corollary}
\newtheorem{rmk}[thm]{Remark}
\newtheorem{definition}[thm]{Definition}

\begin{document}

\makeatletter \@addtoreset{equation}{section}
\def\theequation{\thesection.\arabic{equation}}
\makeatother \makeatletter

\begin{center}
{\Large \bf On Harish-Chandra modules of the Lie algebra arising from the $2$-Dimensional Torus}
\end{center}

\begin{center}
{Zhiqiang Li,
Shaobin Tan$^{a}$\footnote{Partially supported by China NSF grant (Nos.11471268, 11531004).}
and Qing Wang$^{a}$\footnote{Partially supported by
 China NSF grant (Nos.11531004, 11622107), Natural Science Foundation of Fujian Province
	(No.2016J06002) and Fundamental Research Funds for the Central University (No.20720160008).}\\
$\mbox{}^{a}$School of Mathematical Sciences, Xiamen University,
Xiamen 361005, China}
\end{center}

\begin{abstract}
Let $A=\C[t_1^{\pm1},t_2^{\pm1}]$ be the algebra of Laurent polynomials in two variables and
$B$ be the set of skew derivations of $A$. Let $L$ be the universal central extension of the derived Lie subalgebra of the Lie algebra $A\rtimes B$.
Set $\wtL=L\oplus\C d_1\oplus\C d_2$, where $d_1$, $d_2$ are two degree derivations. A Harish-Chandra module is defined as an irreducible weight module with finite dimensional weight spaces.
In this paper, we prove that a Harish-Chandra module of the Lie algebra $\wtL$ is a uniformly bounded module or a generalized highest
weight (GHW for short) module. Furthermore, we prove that the nonzero level Harish-Chandra modules of $\wtL$ are
GHW modules. Finally, we classify all the GHW Harish-Chandra modules of $\wtL$.
\end{abstract}

\section{Introduction}

Harish-Chandra modules (or called irreducible quasifinite weight modules) are no doubt one of the most important family in the study of the representation theory of infinite dimensional
Lie algebras. The complete classification results of Harish-Chandra modules over the Virasoro algebra \cite{KS}, \cite{M}, higher rank Virasoro algebras \cite{Su1}, \cite{LZ1}, and many other Lie algebras related to the Virasoro algebra have been achieved in \cite{GLZ1}, \cite{GLZ2}, \cite{LZ2}, \cite{LJ}, \cite{Maz}, \cite{Su2}, \cite{Su3}, \cite{SXX1}, \cite{SXX2}, \cite{WT} etc. In this paper, we study Harish-Chandra modules over the Lie algebra $\wtL=L\oplus\C d_1\oplus\C d_2$, this Lie algebra can be seen as a generalization of the twisted Heisenberg-Virasoro algebra from rank one to rank two (see \cite{XLT}, \cite{TWX} for details). The structure of the Lie algebra $L$ has been studied in \cite{XLT} ten years ago. Recently, the connection of the Lie algebra $L$ with the vertex algebra has been established in \cite{GW} and the representation theory of the Lie algebra $L$ has been studied in \cite{TWX}, \cite{GL}, \cite{BT}. However, the complete classification result of the Harish-Chandra modules
of the Lie algebra $\wtL$ is still unknown. This paper is contribute to this problem. We prove that a Harish-Chandra module of $\wtL$ is a uniformly bounded module or a GHW module and classify the nonzero level Harish-Chandra modules of the Lie algebra $\wtL$. Based on these results, the complete classification of Harish-Chandra modules of $\wtL$ reduces to the classification of uniformly bounded modules of $\wtL$. In \cite{GL}, the uniformly bounded modules satisfying the condition that the torus subalgebra acting nonzero were classified. Another reason to study the Harish-Chandra modules of the Lie algebra $\wtL$ comes from the representation theory of the nullity $2$ toroidal extended affine Lie algebras of type $A_1$ \cite{CLT}. It was proved therein that the classification of irreducible integrable
modules with finite dimensional spaces of the nullity $2$ toroidal extended affine Lie algebras of type $A_1$ can be reduced to the
classification of Harish-Chandra modules of $\wtL$. This philosophy is similar to that the classification of irreducible integrable
modules of the the full toroidal Lie algebra can be reduced to the
classification of irreducible $(\text{Der}(A_n)\ltimes A_n)$-modules \cite{RJ}, where $A_n=\C[t_1^{\pm1},\ldots, t_n^{\pm1}]$.

The techniques in this paper are following from \cite{LT1}, \cite{LS}, \cite{LZ1}, \cite{Su1}. However, we want to point out that in \cite{LT1}, the construction of the GHW modules of the Virasoro-like algebra are induced from the $\Z$-graded irreducible modules of a Heisenberg subalgebra. While in this paper, the construction of the GHW module of the Lie algebra $\wtL$ comes from the $\Z$-graded irreducible module of the subalgebra $\mathcal{H}_{\bm{b_{1}}}$ (see the definition in Section 2), which is the twist of three Heisenberg subalgebras. So we first need to classify the $\Z$-graded irreducible $\mathcal{H}_{\bm{b_{1}}}$-modules with finite dimensional graded spaces, which are done in the Proposition \ref{zlimoha} and the Proposition \ref{zgimfh}. For the classification of GHW Harish-Chandra modules of $\wtL$, we achieve this by considering the highest weight modules of the Lie algebra $L$ tensor a torus. Moreover, we prove that these tensor product modules of $\wtL$ are completely reducible, and the GHW Harish-Chandra modules of $\wtL$ are isomorphic to the irreducible components of these tensor product modules.

The paper is organized as follows. In Section 2, we prove that
a Harish-chandra module of $\wtL$ is a uniformly bounded module or a GHW module.
In Section 3, we prove that a nonzero level Harish-Chandra module of $\wtL$ is a GHW module. Then we characterize the GHW Harish-Chandra modules with nonzero level. In Section 4, we classify
the GHW Harish-Chandra modules of $\wtL$.

Throughout this paper we use $\C$,
$\Z$, $\Z_+$, $\N$ to denote
the sets of complex numbers, integers, nonnegative integers and positive integers respectively.
All the vector spaces mentioned in this paper are over $\C$. As usual,
if $u_{1}$, $u_{2}$, $\cdots$, $u_{k}$ are elements on some vector space, we
denote $\<u_{1}, u_{2}, \cdots u_{k}\>$ the linear span of the elements $u_{1}$, $u_{2}$, $\cdots$, $u_{k}$ over $\C$.
The universal enveloping algebra for a Lie algebra $\mathfrak{g}$ is denoted by $\U(\mathfrak{g})$ and
$GL_{n\times n}(\Z)$ denotes the set of $n\times n$ invertible matrices with entries in $\Z$.

\section{Harish-Chandra modules of $\wtL$}
In this section, we first recall some basic definitions about Harish-Chandra modules of $\wtL$ and some results for Heisenberg algebras. Then we prove that a Harish-Chandra module of $\wtL$ is a uniformly bounded module or a generalized highest weight module.

Let $\bm{e_{1}}=(1, 0)$,
$\bm{e_{2}}=(0, 1)$, $\Gamma=\Z\bm{e_{1}}+\Z\bm{e_{2}}$.
Let $(x_{1}, x_{2}), (y_{1}, y_{2})\in\Gamma$, we define $(x_{1}, x_{2})>(y_{1}, y_{2})$ if
and only if $x_{1}>y_{1}$ and $x_{2}>y_{2}$; $(x_{1}, x_{2})\geq(y_{1}, y_{2})$ if and only if
$x_{1}\geq y_{1}$ and $x_{2}\geq y_{2}$. For any $\bm{b_{1}}=b_{11}\bm{e_{1}}+b_{12}\bm{e_{2}}$,
$\bm{b_{2}}=b_{21}\bm{e_{1}}+b_{22}\bm{e_{2}}\in \Gamma$,
we set \[\text{det}{\bm{b_{1}} \choose \bm{b_{2}}}=b_{11}b_{22}-b_{12}b_{21}.\]
We recall the definition
of the Lie algebra arising from the $2$-Dimensional Torus (or called the Heisenberg-Virasoro algebra of rank two). See $\cite{XLT}$ (c.f.\cite{TWX}) for details.

\begin{definition}\label{defHVart}
{\em The {\em Heisenberg-Virasoro algebra of rank two} is the Lie algebra spanned by
\[\{t^{\bm{m}}, E(\bm{m}), K_{i}\mid\bm{m}\in\Gamma\setminus\{\bm{0}\},
i=1, 2, 3, 4\}\]} with Lie bracket defined by
\[[t^{\bm{m}}, t^{\bm{n}}]=0,\ [K_{i}, L]=0, i=1, 2, 3, 4;\]
\[[t^{\bm{m}}, E(\bm{n})]=\text{det}{\bm{n} \choose \bm{m}}
t^{\bm{m}+\bm{n}}+\delta_{\bm{m}+\bm{n}, 0}h(\bm{m});\]
\[[E(\bm{m}), E(\bm{n})]=\text{det}{\bm{n} \choose \bm{m}}
E(\bm{m}+\bm{n})+\delta_{\bm{m}+\bm{n}, 0}f(\bm{m}),\]
$\bm{m}=m_{1}\bm{e_{1}}+m_{2}\bm{e_{2}}$ , where
$h(\bm{m})=m_{1}K_{1}+m_{2}K_{2}$, $f(\bm{m})=m_{1}K_{3}+m_{2}K_{4}$.
\end{definition}
We denote this Lie algebra by $L$. Set $E(\bm{0})=t^{\bm{0}}=0$ for convenience. Obviously the subalgebra
$\<E(\bm{m}),\ K_{3},\ K_{4}\mid\bm{m}\in\Gamma\setminus\{\bm{0}\}\>$ of $L$
is a Virasoro-like algebra. It is obvious that $L$ is a
$\Z^2$-graded Lie algebra. Let $\wtL=L\oplus\C d_{1}\oplus\C d_{2}$, where $d_1,d_2$ are defined by
\[[d_i,E(\bm{m})]=m_iE(\bm{m}),\ [d_i,t^{\bm{m}}]=m_it^{\bm{m}},\ [d_{i},K_{j}]=0,\ [d_1,d_2]=0\]
for $\bm{m}=m_{1}\bm{e_{1}}+m_{2}\bm{e_{2}}\in \Gamma$, $i=1,2$ and $j=1,2,3,4$.
The following lemma is easy to check.

\begin{lem}
Let $0\neq \bm{b_{1}}=b_{11}\bm{e_{1}}+b_{12}\bm{e_{2}}\in\Gamma$
and $\bm{b_{2}}=b_{21}\bm{e_{1}}+b_{22}\bm{e_{2}}\in\Gamma$.\\
(1) $\<E(\pm k\bm{b_{1}}),\ f(\bm{b_{1}})\mid k\in\N\>$,
$\<E(k\bm{b_{1}}),\ t^{-k\bm{b_{1}}},\ h(\bm{b_{1}})\mid k\in\N\>$ and
$\<E(-k\bm{b_{1}}),\ t^{k\bm{b_{1}}},\ h(\bm{b_{1}})\mid k\in\N\>$ are three Heisenberg subalgebras of $\wtL$.\\
(2) $\{\bm{b_{1}},\ \bm{b_{2}}\}$ is a $\Z$-basis of $\Gamma$ if and only if
$\text{det}{\bm{b_{1}} \choose \bm{b_{2}}}=\pm1$.
\end{lem}

Now we recall some definitions related to the Harish-Chandra modules for $\wtL$.
A {\em weight module} of $\wtL$ is a module $V$ with weight space decomposition:
\[V=\oplus_{\bm{\lambda}\in \C^6}V_{\bm{\lambda}},\]
where $V_{\bm{\lambda}}=\{v\in V\mid d_{i}v=\lambda_{i}v,\ K_{j}v=\lambda_{j+2}v,\ i=1,2,\ j=1,2,3,4\}$
and $\bm{\lambda}=(\lambda_{1},\cdots,\lambda_{6})\in\C^{6}$. A weight module is called {\em quasi-finite} if
all weight spaces $V_{\bm{\lambda}}$ are finite dimensional. Furthermore, if there exists a positive integer $N$
such that dim $V_{\bm{\lambda}}\leq N$ for all $\bm{\lambda}\in\C^6$, we call such modules are {\em uniformly bounded modules}.
An irreducible quasi-finite weight module is called a {\em Harish-Chandra module}.
Note that the centers $K_{1},K_{2},K_{3},K_{4}$ of $\wtL$ act on an irreducible
weight module $V$ as scalars, i.e., $K_{i}.v=c_{i}v$ for all $v\in V$, $c_i\in\C$, $i=1,2,3,4$.
And we call the ordered number $(c_{1}, c_{2}, c_{3}, c_{4})$ the {\em level} of the module $V$.
We write $V_{(\lambda_{1}, \lambda_{2})}$ instead of $V_{(\lambda_{1},\cdots,\lambda_{6})}$ if the level is fixed.
For a weight module $V$, we define the weight set of $V$ by $\mathcal{P}(V)=\{\bm{\lambda}\in \C^{2}\mid V_{\bm{\lambda}}\neq0\}$.
One can easily see that there exist $\lambda_{1},
\lambda_{2}\in\C$ such that $\mathcal{P}(V)\subseteq(\lambda_{1}, \lambda_{2})+\Gamma$ for an irreducible
$\wtL$-module $V$. If there exists a $\Z$-basis $B=\{\bm{b_{1}}, \bm{b_{2}}\}$ of $\Gamma$ and
$0\neq v_{\bm{\lambda}}\in V_{\bm{\lambda}}$ such that $V=\mathcal{U}(\wtL)v_{\bm{\lambda}}$
and $E(\bm{m})v_{\bm{\lambda}}=t^{\bm{m}}v_{\bm{\lambda}}=0$,
$\forall\bm{m}\in \Z_{+}\bm{b_{1}}+\Z_{+}\bm{b_{2}}$, we call $V$ a {\em generalized
highest weight (GHW for short) module} with GHW $\bm{\lambda}$ corresponding to the $\Z$-basis $B$. The nonzero
vector $v_{\bm{\lambda}}$ is called a {\em GHW vector corresponding to the $\Z$-basis $B$}, or simply {\em GHW vector}.

Let $\{\bm{b_{1}}, \bm{b_{2}}\}$ be a $\Z$-basis of $\Gamma$ and let
$\mathcal{H}_{\bm{b_{1}}}=\<E(k\bm{b_{1}}),\ t^{k\bm{b_{1}}},\ K_i\mid k\in\Z\setminus\{0\},\ i=1,2,3,4\>$. Denote
\[\wtL_0=\mathcal{H}_{\bm{b_{1}}}\oplus\C d_{1}\oplus\C d_{2},\]
\[\wtL_i=\<E(m\bm{b_{1}}+i\bm{b_{2}}), t^{m\bm{b_{1}}+i\bm{b_{2}}}\mid m\in\Z\>,\  i\neq0,\]
\[\wtL_+=\oplus_{i>0}\wtL_i,\ \wtL_{-}=\oplus_{i<0}\wtL_i.\]
Then $\wtL=\wtL_{+}\oplus\wtL_{0}\oplus\wtL_{-}$. Let $V$ be an irreducible weight $\wtL_{0}$-module.
We extend the $\wtL_{0}$-module structure on $V$ to a $(\wtL_{+}\oplus\wtL_{0})$-module structure by
defining $\wtL_{+}.V=0$. Then we obtain the induced $\wtL$-module
\[\wtM(V)=\wtM(\bm{b_{1}}, \bm{b_{2}}, V)=\text{Ind}_{\wtL_{+}\oplus\wtL_{0}}^{\wtL}V
=\U(\wtL)\otimes_{\U(\wtL_{+}\oplus\wtL_{0})}V.\]
It is clear that, as vector spaces, $\wtM(\bm{b_{1}}, \bm{b_{2}}, V)\simeq\U(\wtL_{-})\otimes_{\C}V$.
The $\wtL$-module $\wtM(\bm{b_{1}}, \bm{b_{2}}, V)$ has a unique maximal submodule
$J(\bm{b_{1}}, \bm{b_{2}}, V)$ trivially intersecting with $V$. Then we obtain the unique irreducible quotient module
\[M(V)=M(\bm{b_{1}}, \bm{b_{2}}, V)=
\wtM(\bm{b_{1}}, \bm{b_{2}}, V)/J(\bm{b_{1}}, \bm{b_{2}}, V).\]
It is clear that $M(V)$ is uniquely determined by the $\Z$-basis $\{\bm{b_{1}}, \bm{b_{2}}\}$ of
$\Gamma$ and the $\wtL_{0}$-module $V$.

\begin{rmk}
The irreducible $\wtL$-module $M(\bm{b_{1}}, \bm{b_{2}}, V)$ constructed above is
a GHW module corresponding to the $\Z$-basis $\{\bm{b_1}+\bm{b_2},\bm{b_1}+2\bm{b_2}\}$ of $\Gamma$.
\end{rmk}

Now we recall some results about the $\Z$-graded module for Heisenberg Lie algebras.

For any $0\neq\bm{b_{1}}\in\Gamma$, denote the subalgebra $\<E(\pm k\bm{b_{1}}),\ f(\bm{b_{1}})\mid k\in\N\>$
of $\wtL$ by $E_{\bm{b_{1}}}$. For any $E_{\bm{b_{1}}}$-module $V$,
if the eigenvalue of $f(\bm{b_{1}})$ is a scalar then we call it the {\em level} of $V$. Let
\[E_{\bm{b_{1}}}^{\pm}=\<E(k\bm{b_{1}})\mid\pm k\in \N\>.\]
For $0\neq a\in \C$, let $\C v_{a}$ be a one dimensional
$(E_{\bm{b_{1}}}^{\varepsilon}\oplus \C f(\bm{b_{1}}))$-module such that
$E_{\bm{b_{1}}}^{\varepsilon}.v_{a}=0$, $f(\bm{b_{1}}).v_{a}=av_{a}$, $\varepsilon\in \{+, -\}$. Consider the induced $E_{\bm{b_{1}}}$-module
\[M^{\varepsilon}(a)
=\U(E_{\bm{b_{1}}})\otimes_{\U(E_{\bm{b_{1}}}^{\varepsilon}\oplus\C f(\bm{b_{1}}))}\C v_{a}\]
associated with the $a$ and $\varepsilon$ ($a$ is the level of $M^{\varepsilon}(a)$). Then the $E_{\bm{b_{1}}}$-module
$M^{\varepsilon}(a)$ is irreducible.

The following results was due to the Propositions $4.3(\mbox{i})$ and the Proposition $4.5$ in \cite{F}.

\begin{thm}\label{nzlham}
If $V=\oplus_{i\in\Z}V_{i}$ is a $\Z$-graded $E_{\bm{b_{1}}}$-module
of level $0\neq a\in\C$ and dim $V_{i}<\infty$ for at least one $i\in\Z$ then\\
(1). If $V$ is an irreducible module then $V\simeq M^{\varepsilon}(a)$ for some $\varepsilon\in\{+, -\}$;\\
(2). $V$ is completely reducible.
\end{thm}

Let $\{\bm{b_{1}}, \bm{b_{2}}\}$ be a $\Z$-basis of $\Gamma$. For a $\mathcal{H}_{\bm{b_{1}}}$-module $V$,
if $f(\bm{b_1}),h(\bm{b_1}),f(\bm{b_2}),h(\bm{b_2})$
act as scalars $c_1,c_2,c_3,c_4\in\C$, then we call $(c_1,c_2,c_3,c_4)$ the {\em level} of the $\mathcal{H}_{\bm{b_{1}}}$-module $V$.
Furthermore if $(c_1,c_2,c_3,c_4)=(0,0,c_3,c_4)$, we say that $V$ is a $\mathcal{H}_{\bm{b_{1}}}$-module of {\em level zero}. Otherwise,
$V$ is nonzero level. In the following, we will discuss the irreducible $\mathcal{H}_{\bm{b_{1}}}$-modules.
First we recall the classification of $\Z$-graded irreducible $\mathcal{H}_{\bm{b_{1}}}$-modules of level zero.
Then we classify the $\Z$-graded $\mathcal{H}_{\bm{b_{1}}}$-modules of nonzero level with finite-dimensional graded subspaces.

Set $T=\C[t^{\pm1}]$. Let $\rho:\mathcal{H}_{\bm{b_{1}}}\rightarrow\C$ be a linear function with $\rho(f(\bm{b_1}))=\rho(h(\bm{b_1}))=0$.
We can define a $\mathcal{H}_{\bm{b_{1}}}$-module structure on $T$ by
\begin{eqnarray}f(\bm{b_1}).t^n=0,\ E(k\bm{b_1}).t^n=\rho(E(k\bm{b_1}))t^{k+n}\label{zgm}\\
h(\bm{b_1}).t^n=0,\ t^{k\bm{b_1}}.t^n=\rho(t^{k\bm{b_1}})t^{k+n},\\
f(\bm{b_2}).t^n=\rho(f(\bm{b_2}))t^n,\ h(\bm{b_2}).t^n=\rho(h(\bm{b_2}))t^n,
\end{eqnarray}
$n\in\Z$, $k\in\Z\setminus\{0\}$. We denote \[T_{\rho,i}(\mathcal{H}_{\bm{b_{1}}})=\U(\mathcal{H}_{\bm{b_{1}}}).t^i\] be the
$\mathcal{H}_{\bm{b_{1}}}$-submodule of $T$ generated by $t^i$ for $i\in\Z$.
And we denote $T_{\rho,0}(\mathcal{H}_{\bm{b_{1}}})$ by $T_{\rho}(\mathcal{H}_{\bm{b_{1}}})$ for short. From the definition,
we see that \begin{align}T_{\rho,i}(\mathcal{H}_{\bm{b_{1}}})\simeq T_{\rho,j}(\mathcal{H}_{\bm{b_{1}}})\end{align} for $i,j\in\Z$ as
$\mathcal{H}_{\bm{b_{1}}}$-modules.

\begin{rmk}
For linear function $\rho: E_{\bm{b_{1}}}\rightarrow\C$ with $\rho(f(\bm{b_1}))=0$, we can define
a $E_{\bm{b_{1}}}$-module structure on the Laurent polynomial ring $T$ with the action
given by (\ref{zgm}). Similarly, let $T_{\rho,i}(E_{\bm{b_1}}):=\U(E_{\bm{b_{1}}}).t^i$ be
the $E_{\bm{b_{1}}}$-submodule of $T$ generated by $t^i$ for $i\in\Z$.
And we also write $T_{\rho,0}(E_{\bm{b_1}})$ by $T_{\rho}(E_{\bm{b_1}})$ for short.
\end{rmk}

Then we have the following results from the Lemma $3.6$ and the Proposition $3.8$ in \cite{C}.

\begin{prop}\label{zlimoha}
(1). The $\mathcal{H}_{\bm{b_{1}}}$-module $T_{\rho}(\mathcal{H}_{\bm{b_{1}}})$
(resp. $E_{\bm{b_{1}}}$-module $T_{\rho}(E_{\bm{b_1}})$) is irreducible if and only if $T_{\rho}(\mathcal{H}_{\bm{b_{1}}})=T_r$
(resp. $T_{\rho}(E_{\bm{b_1}})=T_r$) for some $r\in\Z_+$, where $T_0=\C1$ and $T_r=\C[t^r, t^{-r}]$ if $r\in\N$.\\
(2). If $V$ is a $\Z$-graded irreducible $\mathcal{H}_{\bm{b_{1}}}$-module (resp. $E_{\bm{b_{1}}}$-module) of level zero,
then $V\simeq T_{\rho}(\mathcal{H}_{\bm{b_{1}}})$ for some
linear function $\rho: \mathcal{H}_{\bm{b_{1}}}\rightarrow\C$ with $\rho(f(\bm{b_1}))=\rho(h(\bm{b_1}))=0$
(resp. $V\simeq T_{\rho}(E_{\bm{b_1}})$ for some linear function $\rho: E_{\bm{b_{1}}}\rightarrow\C$ with $\rho(f(\bm{b_1}))=0$),
and $T_{\rho}(\mathcal{H}_{\bm{b_{1}}})=T_r$ (resp. $T_{\rho}(E_{\bm{b_1}})=T_r$) for some $r\in\Z_+$.
\end{prop}

\begin{rmk}
Since $\<t^{k\bm{b_1}},E(-k\bm{b_1}),h(\bm{b_1})\mid k\in\N\>$ and
$\<t^{-k\bm{b_1}},E(k\bm{b_1}),h(\bm{b_1})\mid k\in\N\>$ are two Heisenberg Lie subalgebras of $\wtL$.
The Theorem \ref{nzlham} and the Proposition \ref{zlimoha} also hold for their corresponding $\Z$-graded irreducible modules.
\end{rmk}

For convenience, we denote $\mathcal{E}_{\bm{b_1}}$ the set of all linear functions $\rho:\mathcal{H}_{\bm{b_{1}}}\rightarrow\C$
with $\rho(f(\bm{b_1}))=\rho(h(\bm{b_1}))=0$ such that the $\mathcal{H}_{\bm{b_{1}}}$-module
$T_{\rho}(\mathcal{H}_{\bm{b_{1}}})$ is irreducible.
Let $t_{\bm{b_1}}=\<t^{\pm k\bm{b_1}}\mid k\in\N\>$. Then $t_{\bm{b_1}}$ is a subalgebra of $\wtL$.

\begin{prop}\label{zgimfh}
Let $V=\bigoplus_{i\in\Z}V_{i}$ be a $\Z$-graded irreducible $\mathcal{H}_{\bm{b_{1}}}$-module
with $\text{dim}\ V_{i}<\infty$ for all $i\in\Z$. Suppose that $f(\bm{b_{1}}).v=c_{1}v$,
$h(\bm{b_{1}}).v=c_{2}v$, $f(\bm{b_{2}}).v=c_{3}v$ and $h(\bm{b_{2}}).v=c_{4}v$ for $v\in V$,
where $c_1, c_2,c_3,c_4\in\C$ and $(c_{1}, c_{2})\neq\bm{0}$.\\
(1). If $c_{1}\neq0$ and $c_{2}\neq0$, then \[V\simeq \U(\mathcal{H}_{\bm{b_{1}}})
\otimes_{\U(\<E(k\bm{b_{1}}),\ t^{k\bm{b_{1}}},\ f(\bm{b_{i}}),\ h(\bm{b_{i}})\mid k\in\N,\ i=1,2\>)}\C1,\]
where $\<E(k\bm{b_{1}}),\ t^{k\bm{b_{1}}}\mid k\in \N\>.1=0$, $f(\bm{b_{1}}).1=c_11$, $h(\bm{b_{1}}).1=c_21$,
$f(\bm{b_{2}}).1=c_31$ and $h(\bm{b_{2}}).1=c_41$ or
\[V\simeq \U(\mathcal{H}_{\bm{b_{1}}})
\otimes_{\U(\<E(-k\bm{b_{1}}),\ t^{-k\bm{b_{1}}},\ f(\bm{b_{i}}),\ h(\bm{b_{i}})\mid k\in \N,\ i=1,2\>)}\C1,\]
where $\<E(-k\bm{b_{1}}),\ t^{-k\bm{b_{1}}}\mid k\in \N\>.1=0$, $f(\bm{b_{1}}).1=c_11$, $h(\bm{b_{1}}).1=c_21$,
$f(\bm{b_{2}}).1=c_31$ and $h(\bm{b_{2}}).1=c_41$ .\\
(2). If $c_{1}\neq0$ and $c_{2}=0$, then \[V\simeq T_{\rho}(t_{\bm{b_1}})\otimes M^{\varepsilon}(c_{1}),\]
for some linear function $\rho: t_{\bm{b_1}}\rightarrow\C$,
such that $T_{\rho}(t_{\bm{b_1}})=T_r$ for some $r\in\Z_+$,
where $M^{\varepsilon}(c_{1})$ is the irreducible $E_{\bm{b_{1}}}$-module of level $c_1$, $\varepsilon\in\{+, -\}$.\\
(3). If $c_{1}=0$ and $c_{2}\neq0$, then \[V\simeq \U(\mathcal{H}_{\bm{b_{1}}})
\otimes_{\U(\<E(k\bm{b_{1}}),\ t^{k\bm{b_{1}}},\ f(\bm{b_{i}}),\ h(\bm{b_{i}})\mid k\in \N,\ i=1,2\>)}\C1,\]
where $\<E(k\bm{b_{1}}),\ t^{k\bm{b_{1}}}\mid k\in \N\>.1=0$, $f(\bm{b_1}).1=0$, $h(\bm{b_1}).1=c_21$,
$f(\bm{b_{2}}).1=c_31$ and $h(\bm{b_{2}}).1=c_41$ or
\[V\simeq \U(\mathcal{H}_{\bm{b_{1}}})
\otimes_{\U(\<E(-k\bm{b_{1}}),\ t^{-k\bm{b_{1}}},\ f(\bm{b_{i}}),\ h(\bm{b_{i}})\mid k\in \N,\ i=1,2\>)}\C1,\]
where $\<E(-k\bm{b_{1}}),\ t^{-k\bm{b_{1}}}\mid k\in \N\>.1=0$, $f(\bm{b_1}).1=0$, $h(\bm{b_1}).1=c_21$,
$f(\bm{b_{2}}).1=c_31$ and $h(\bm{b_{2}}).1=c_41$.
\end{prop}

\begin{proof}
(1). If $c_{1}\neq0$ and $c_{2}\neq0$, by the Theorem \ref{nzlham}, we know that there exists some
$0\neq v_{0}\in V_{i_{0}}$ for some $i_{0}\in\Z$ such that $E(k\bm{b_{1}}).v_{0}=0$ for any $k\in \N$ or $-k\in\N$.
Without loss of generality, we assume $k\in \N$,
then $\U(\<E(-k\bm{b_{1}})\mid k\in\N\>)v_{0}$ is an irreducible $E_{\bm{b_{1}}}$-module. Let
\[W:=\U(\<t^{k\bm{b_{1}}},\ E(l\bm{b_{1}}),\ f(\bm{b_{1}}),\
h(\bm{b_{1}})\mid k\in \N,\ l\in \Z\setminus\{0\}\>)v_{0}\subseteq V\]
Note that $W$ as a $\Z$-graded $\<t^{k\bm{b_{1}}},\ E(-k\bm{b_{1}}),\ h(\bm{b_{1}})\mid k\in \N\>$-module is completely reducible,
Then we have that \[W=\big(\bigoplus_{i\in I}(\bigoplus_{m_i\in X_i}V^{+}_{i,m_i})\big)\oplus\big
(\bigoplus_{j\in J}(\bigoplus_{n_j\in Y_j}V_{j,n_j}^{-})\big),\]
where \[V_{i,m_i}^{+}=\U(\<t^{k\bm{b_{1}}},\ E(-k\bm{b_{1}}),\ h(\bm{b_{1}})\mid k\in \N\>)v_{i,m_i}\simeq M^{+}(c_{2}),\]
for some $0\neq v_{i,m_i}\in V_i\cap W$ with $t^{k\bm{b_{1}}}.v_{i,m_i}=0$ for all $k\in\N$, $i\in I$, $m_i\in X_i$,
\[V_{j,n_j}^{-}=\U(\<t^{k\bm{b_{1}}},\ E(-k\bm{b_{1}}),\ h(\bm{b_{1}})\mid k\in \N\>)u_{j,n_j}\simeq M^{-}(c_{2}),\]
where $0\neq u_{j,n_j}\in V_j\cap W$ with $E(-k\bm{b_{1}}).u_{j,n_j}=0$ for all $k\in\N$, $j\in J$, $n_j\in Y_j$, $I, J, X_i, Y_j\subseteq \Z$.
Note that $I$ has an upper bound, $J$ has a lower bound and
all $X_i,Y_j$ are finite sets since dim $V_n<\infty$ for all $n\in\Z$.
Assume $J\neq\emptyset$, then there exists some nonzero vector $w_{0}\in W\cap V_{i}$ such that $E(-k\bm{b_{1}}).w_{0}=0$ for all $k\in\N$
and some $i\in \Z$. Consider $W_{0}=\U(E_{\bm{b_1}})w_{0}\subseteq W$,
then \[W_{0}=\U(\<E(l\bm{b_{1}})\mid l\in\N\>).w_{0}\] and $W_{0}$
is a free $\mathcal{U}(\<E(l\bm{b_{1}})\mid l\in\N\>)$-module. On the other hand,
since $w_{0}\in W$, then there exists $k\in \N$ such that $E(k\bm{b_{1}}).w_{0}=0$, which is a contradiction.
Thus $J=\emptyset$ and $W=\bigoplus_{i\in I}(\bigoplus_{m_i\in X_i}V^{+}_{i,m_i})$.
Since $I$ has an upper bound, then there exists $0\neq u_{0}\in W\cap V_{i_1}$ for some $i_1\in\Z$
such that $E(k\bm{b_{1}}).u_{0}=t^{k\bm{b_{1}}}.u_{0}=0$
for all $k\in\N$. This shows that
\[V\simeq \U(\mathcal{H}_{\bm{b_{1}}})
\otimes_{\U(\<E(k\bm{b_{1}}),\ t^{k\bm{b_{1}}},\ f(\bm{b_{i}}),\ h(\bm{b_{i}})\mid k\in\N,\ i=1,2\>)} \C u_0.\] Another case is similar.\\
(2). If $c_{1}\neq0$ and $c_{2}=0$, we can write
\[\mathcal{H}_{\bm{b_{1}}}=t_{\bm{b_1}}\oplus E_{\bm{b_1}}\oplus\C f(\bm{b_2})\oplus\C h(\bm{b_1})\oplus\C h(\bm{b_2}).\]
From the Theorem \ref{nzlham}, the Proposition \ref{zlimoha} and the Lemma $2.7$ in \cite{L}, the result follows.\\
(3). If $c_{1}=0$ and $c_{2}\neq0$, by the Theorem \ref{nzlham}, $V$ is completely reducible when we view $V$ as the module of the two
subalgebras $\<t^{-k\bm{b_{1}}},\ E(k\bm{b_{1}}),\ h(\bm{b_{1}})\mid k\in \N\>$ and
$\<t^{k\bm{b_{1}}},\ E(-k\bm{b_{1}}),\ h(\bm{b_{1}})\mid k\in \N\>$. We write
\[V=\big(\bigoplus_{i\in I}(\bigoplus_{m_i\in X_i}V^{+}_{i,m_i})\big)\oplus\big
(\bigoplus_{j\in J}(\bigoplus_{n_j\in Y_j}V_{j,n_j}^{-})\big)\]
when viewed as the module of the Lie algebra $\<t^{-k\bm{b_{1}}},\ E(k\bm{b_{1}}),\ h(\bm{b_{1}})\mid k\in \N\>$,
where $I,J,X_i,Y_j\subseteq\Z$, \[V^{+}_{i,m_i}=\U(\<t^{-k\bm{b_{1}}},\ E(k\bm{b_{1}}),\ h(\bm{b_{1}})\mid k\in \N\>)v_{i,m_i}\simeq M^{+}(c_2)\]
with $E(k\bm{b_{1}}).v_{i,m_i}=0$ for all $k\in\N$, $i\in I$, $m_i\in X_i$ and $0\neq v_{i,m_i}\in V_i$,
\[V^{-}_{j,n_j}=\U(\<t^{-k\bm{b_{1}}},\ E(k\bm{b_{1}}),\ h(\bm{b_{1}})\mid k\in \N\>)u_{j,n_j}\simeq M^{-}(c_2)\]
with $t^{-k\bm{b_{1}}}.u_{j,n_j}=0$ for all $k\in\N$, $j\in J$, $n_j\in Y_j$ and $0\neq u_{j,n_j}\in V_j$. Similarly, we write
\[V=\big(\bigoplus_{i\in I^{\prime}}(\bigoplus_{p_i\in X_i^\prime}W^{+}_{i,p_i})\big)
\oplus\big(\bigoplus_{j\in J^{\prime}}(\bigoplus_{q_j\in Y_j^\prime}W_{j,q_j}^{-})\big)\]
when viewed as the module of the Lie algebra $\<t^{k\bm{b_{1}}},\ E(-k\bm{b_{1}}),\ h(\bm{b_{1}})\mid k\in \N\>$.
Note that both $I$ and $I^\prime$ have upper bounds, $J$ and $J^\prime$ have lower bounds and
all $X_i,Y_j,X_i^\prime,Y_j^\prime$ are finite sets as dim $V_n<\infty$ for all $n\in\Z$.
If $I=\emptyset$, similar to the proof in $(1)$, we get \[V\simeq \U(\mathcal{H}_{\bm{b_{1}}})
\otimes_{\U(\<E(k\bm{b_{1}}),\ t^{k\bm{b_{1}}},\ f(\bm{b_{i}}),\ h(\bm{b_{i}})\mid k\in \N,\ i=1,2\>)}\C1,\]
where $\<E(k\bm{b_{1}}),\ t^{k\bm{b_{1}}}\mid k\in \N\>.1=0$, $f(\bm{b_1}).1=0$, $h(\bm{b_1}).1=c_21$,
$f(\bm{b_{2}}).1=c_31$ and $h(\bm{b_{2}}).1=c_41$. Now suppose $I\neq\emptyset$,
we can choose $0\neq v_{0}\in V^{+}_{i_0,m_{i_0}}$ for some
$i_0\in I$, $m_{i_0}\in X_{i_0}$ such that $E(k\bm{b_{1}}).v_{0}=0$ for all $k\in\N$.
Then we have $v_{0}=w_{1}+w_{2}$, where $w_{1}\in\big(\bigoplus_{i\in I^{\prime}}(\bigoplus_{p_i\in X_i^\prime}W^{+}_{i,p_i})\big)\cap V_{i_{0}}$ and
$w_{2}\in\big(\bigoplus_{j\in J^{\prime}}(\bigoplus_{q_j\in Y_j^\prime}W_{j,q_j}^{-})\big)\cap V_{i_{0}}$.
If $w_{1}\neq0$, we can choose large enough $k_0\in\N$ such that $E(-k_0\bm{b_1}).w_{2}=0$ since $J^\prime$ has a lower bound.
Since $\bigoplus_{i\in I^{\prime}}(\bigoplus_{p_i\in X_i^\prime}W^{+}_{i,p_i})$ is a free $\<E(-k\bm{b_{1}})\mid k\in\N\>$-module,
then $0\neq E(-k_0\bm{b_1}).w_{1}=E(-k_0\bm{b_1}).v_{0}\in\bigoplus_{i\in I^{\prime}}(\bigoplus_{p_i\in X_i^\prime}W^{+}_{i,p_i})$
and $E(k\bm{b_{1}}).E(-k_0\bm{b_1}).v_{0}=E(-k_0\bm{b_1}).E(k\bm{b_{1}}).v_{0}=0$ for all $k\in\N$. Now we claim that there exists
$0\neq v\in V_{i}$ for some $i\in\Z$ such that $t^{k\bm{b_{1}}}.v=E(k\bm{b_{1}}).v=0$
for all $k\in\N$. In fact, if $t^{k\bm{b_1}}.(E(-k_0\bm{b_1}).v_{0})=0$ for all $k\in\N$, this is done by setting $v=E(-k_0\bm{b_1}).v_{0}$.
If there exists $k_1\in\N$ such that $t^{k_1\bm{b_1}}.(E(-k_0\bm{b_1}).v_0)\neq0$, set $v_1=t^{k_1\bm{b_1}}.(E(-k_0\bm{b_1}).v_0)$.
Repeating the process, since $I^\prime$ has an upper bound, we know that this process will terminate after finite steps.
This implies that \[V\simeq \U(\mathcal{H}_{\bm{b_{1}}})
\otimes_{\U(\<E(k\bm{b_{1}}),\ t^{k\bm{b_{1}}},\ f(\bm{b_{i}}),\ h(\bm{b_{i}})\mid k\in \N,\ i=1,2\>)}\C1,\]
where $\<E(k\bm{b_{1}}),\ t^{k\bm{b_{1}}}\mid k\in \N\>.1=0$, $f(\bm{b_1}).1=0$, $h(\bm{b_1}).1=c_21$,
$f(\bm{b_{2}}).1=c_31$ and $h(\bm{b_{2}}).1=c_41$. If $w_{1}=0$, i.e., $v_0=w_2$, we know that there exists some
$0\neq u\in V_{j_0}$ for some $j_0\in\Z$ such that
$E(k\bm{b_1}).u=0$ for $k\in\Z\setminus\{0\}$. In fact, if there exists $n_1\in\N$ such that $E(-n_1\bm{b_1})v_0\neq0$, set
$u_1=E(-n_1\bm{b_1})v_0$. We also have $E(k\bm{b_1}).u_1=0$ for all $k\in\N$ since $f(\bm{b_1}).V=0$.
Repeating the process, since $J^\prime$ has a lower bound,
we know that this process will terminate after finite steps. Then
\[V\simeq\U(\mathcal{H}_{\bm{b_{1}}})\otimes_{\U(\<E(m\bm{b_1}),\ f(\bm{b_1}),\ h(\bm{b_1})\mid m\in\Z\setminus\{0\}\>)}\C1,\]
where $E(m\bm{b_1}).1=0$ for all $m\in\Z\setminus\{0\}$, $f(\bm{b_1}).1=0$, $h(\bm{b_1}).1=c_21$, $f(\bm{b_{2}}).1=c_31$ and $h(\bm{b_{2}}).1=c_41$.
This contradicts to the condition that dim $V_{i}<\infty$ for all $i\in\Z$. Then the conclusion follows.
\end{proof}

Fix a $\Z$-basis $\{\bm{b_{1}}, \bm{b_{2}}\}$ of $\Gamma$,
$\bm{b_{1}}=b_{11}\bm{e_{1}}+b_{12}\bm{e_{2}}$,
and $\lambda_{1}, \lambda_{2}\in \C$. Any $\Z$-graded $\mathcal{H}_{\bm{b_{1}}}$-module
$V=\oplus_{i\in \Z}V_{i}$ with fixed level can be extended to a weight module of $\wtL_{0}$ by defining
\[d_{1}v_{j}=(\lambda_{1}+jb_{11})v_{j},\ d_{2}v_{j}=(\lambda_{2}+jb_{12})v_{j},\]
for $v_{j}\in V_{j}$, $j\in\Z$. One can easily see that the vector space $V$ is a $\wtL_{0}$-module
and $\mathcal{P}(V)\subseteq(\lambda_{1}, \lambda_{2})+\Z\bm{b_{1}}$.
For the $\Z$-graded irreducible $\mathcal{H}_{\bm{b_{1}}}$-modules given in the Proposition \ref{zlimoha} and \ref{zgimfh},
we let
\[V^+(\bm{c})=\U(\mathcal{H}_{\bm{b_{1}}})
\otimes_{\U(\<E(k\bm{b_{1}}),\ t^{k\bm{b_{1}}},\ K_i\mid k\in\N,\ i=1,2,3,4\>)}\C1,\]
\[V^{-}(\bm{c})=\U(\mathcal{H}_{\bm{b_{1}}})
\otimes_{\U(\<E(-k\bm{b_{1}}),\ t^{-k\bm{b_{1}}},\ K_i\mid k\in\N,\ i=1,2,3,4\>)}\C1,\]
$M_{\rho}^{\varepsilon}(\bm{c})=T_{\rho}(t_{\bm{b_1}})\otimes M^{\varepsilon}(c_{1})$ and $T_{\rho}(\mathcal{H}_{\bm{b_{1}}})(\bm{c})=T_{\rho}(\mathcal{H}_{\bm{b_{1}}})$,
we can extend those modules to weight $\wtL_{0}$-modules through the above way, then we denote
the corresponding $\wtL_{0}$-module by $V^{+}(\bm{c}, \bm{\lambda})$,
$V^{-}(\bm{c}, \bm{\lambda})$,
$M_{\rho}^{\varepsilon}(\bm{c},\bm{\lambda})$ and
$T_{\rho}(\mathcal{H}_{\bm{b_{1}}})(\bm{c},\bm{\lambda})$ respectively, where $\bm{c}=(c_{1}, c_{2},c_3,c_4)$,
$\bm{\lambda}=(\lambda_{1}, \lambda_{2})$,
$f(\bm{b_1})$, $h(\bm{b_1})$, $f(\bm{b_2})$ and $h(\bm{b_2})$ act as the scalars $c_1$, $c_2$, $c_3$, $c_4\in\C$, respectively.

With the above notations, the following results can be obtained from the Proposition \ref{zlimoha} and \ref{zgimfh}.

\begin{cor}\label{zimoh}
Let $V=\oplus_{i\in \Z}V_{i}$ be any irreducible weight module of $\wtL_{0}$ with dim $V_{i}<\infty$ for all $i\in\Z$,
and $f(\bm{b_{1}}).v=c_{1}v$, $h(\bm{b_{1}}).v=c_{2}v$, $f(\bm{b_2}).v=c_3v$ and $h(\bm{b_2}).v=c_4v$ for $v\in V$, where
$V_{i}:=V_{(\lambda_{1}, \lambda_{2})+i\bm{b_{1}}}$ for the fixed $\bm{\lambda}=(\lambda_1,\lambda_2)\in\C^2$.\\
(1). If $(c_{1}, c_{2})\neq0$, then $V\simeq V^{\varepsilon}(\bm{c}, \bm{\lambda})$ or
$V\simeq M_{\rho}^{\varepsilon}(\bm{c},\bm{\lambda})$ for some linear function
$\rho:t_{\bm{b_1}}\rightarrow\C$ with $T_{\rho}(t_{\bm{b_1}})=T_r$ for some $r\in\Z_+$.\\
(2). If $(c_{1}, c_{2})=0$, then
$V\simeq T_{\rho}(\mathcal{H}_{\bm{b_{1}}})(\bm{c},\bm{\lambda})$ for some $\rho\in\mathcal{E}_{\bm{b_1}}$.
\end{cor}

The following lemma give the characterization of the irreducible weight modules of $\wtL$ with finite dimensional weight spaces.
\begin{lem}\label{izmoL}
Let $\{\bm{b_{1}}, \bm{b_{2}}\}$ be a $\Z$-basis of $\Gamma$. $V$ is an irreducible
weight module of $\wtL$ with finite dimensional weight spaces and $f(\bm{b_1}),h(\bm{b_1}), f({\bm{b_2}}),h(\bm{b_2})$
act on $V$ as scalars $c_1,c_2,c_3,c_4$ respectively. If there exist $\lambda_{1}, \lambda_{2}\in \C$ such that
$V_{(\lambda_{1}, \lambda_{2})}\neq0$ and $\mathcal{P}(V)\cap \big((\lambda_{1}, \lambda_{2})+\Z\bm{b_{1}}+
\N\bm{b_{2}}\big)=\emptyset$, we have\\
(1). if $c_1=c_2=0$, $V\simeq M\big(\bm{b_{1}}, \bm{b_{2}}, T_{\rho}(\mathcal{H}_{\bm{b_{1}}})(\bm{c},\bm{\lambda})\big)$ for some $\rho\in\mathcal{E}_{\bm{b_1}}$,\\
(2). if $c_1\neq0$, $c_2=0$, $V\simeq M\big(\bm{b_{1}}, \bm{b_{2}},M_{\rho}^{\varepsilon}(\bm{c},\bm{\lambda})\big)$ for some linear function
$\rho:t_{\bm{b_1}}\rightarrow\C$ satisfying $T_{\rho}(t_{\bm{b_1}})=T_r$ for some $r\in\Z_+$,\\
(3). if $c_2\neq0$, $V\simeq M\big(\bm{b_{1}}, \bm{b_{2}}, V^{\varepsilon}(\bm{c}, \bm{\lambda})\big)$,\\
where $\varepsilon\in \{+, -\}$, $\bm{\lambda}=(\lambda_{1}, \lambda_{2})$, $\bm{c}=(c_1,c_2,c_3,c_4)$.
\end{lem}

\begin{proof}
Let $W=\oplus_{i\in \Z}V_{(\lambda_{1}, \lambda_{2})+i\bm{b_{1}}}$. Since
$\mathcal{P}(V)\cap \big((\lambda_{1}, \lambda_{2})+\Z\bm{b_{1}}+
\N\bm{b_{2}}\big)=\emptyset$, we see that $W$ is an irreducible $\wtL_{0}$ weight module and $\wtL_{+}W=0$.
Thus by the construction of $\wtM(\bm{b_{1}}, \bm{b_{2}}, W)$ and the PBW theorem, there exists an epimorphism $\varphi$ from
$\wtM(\bm{b_{1}}, \bm{b_{2}}, W)$ to $V$ such that $\varphi\mid_{W}=\text{id}_{W}$.
Therefore, the lemma follows from the Corollary \ref{zimoh} and the irreducibility of $V$.
\end{proof}

With the notations in the Lemma \ref{izmoL}, the following lemma shows that the cases $(2)$ and $(3)$ in the Lemma \ref{izmoL} don't occur.

\begin{lem}\label{nzHCmoL}
For any $\Z$-basis $\{\bm{b_{1}}, \bm{b_{2}}\}$ of $\Gamma$,
$M\big(\bm{b_{1}}, \bm{b_{2}},M_{\rho}^{\varepsilon}(\bm{c},\bm{\lambda})\big)$
or $M\big(\bm{b_{1}}, \bm{b_{2}}, V^{\varepsilon}(\bm{c}, \bm{\lambda})\big)$ is not a Harish-Chandra module.
\end{lem}

\begin{proof}
With the notations in the Lemma \ref{izmoL}, for the case that $f(\bm{b_1})$ acts as the scalar $c_1\neq0$,
the lemma follows from the Lemma 2.6 in \cite{LT1}. So we only need to consider the case $c_1=0$, $c_2\neq0$. Without loss
of generality, we may assume that there exists a weight vector $0\neq v_0\in V^{\varepsilon}(\bm{c}, \bm{\lambda})$
such that $E(k\bm{b_1})v_0=t^{k\bm{b_1}}v_0=0$ and
$E(-k\bm{b_1})v_0\neq0$ and $t^{-k\bm{b_1}}v_0\neq0$ for all $k\in\N$ (see the Proposition \ref{zgimfh}$(3)$).
For any $n\in\N$, we can choose $k_j\in\Z$, $1\leq j\leq n$ with $0<k_1<k_2<\cdots<k_n$ such that
$h(-k_j\bm{b_1}+\bm{b_2})v_0\neq0$ for $1\leq j\leq n$. We claim that
\[\{E(k_j\bm{b_1}-\bm{b_2})t^{-k_j\bm{b_1}}v_0\mid1\leq j\leq n\}
\subseteq M\big(\bm{b_{1}}, \bm{b_{2}}, V^{\varepsilon}(\bm{c}, \bm{\lambda})\big)_{(\lambda_1,\lambda_2)-\bm{b_2}}\]
is a set of linear independent vectors,
thus the conclusion follows. In fact, if $\sum_{j=1}^na_jE(k_j\bm{b_1}-\bm{b_2})t^{-k_j\bm{b_1}}v_0=0$, then
\begin{eqnarray*}0&=&t^{-k_1\bm{b_1}+\bm{b_2}}\sum_{j=1}^na_jE(k_j\bm{b_1}-\bm{b_2})t^{-k_j\bm{b_1}}v_0\\
&=&a_1h(-k_1\bm{b_1}+\bm{b_2})t^{-k_1\bm{b_1}}v_0+\sum_{j=2}^na_j
\text{det}{k_j\bm{b_1}-\bm{b_2} \choose -k_1\bm{b_1}+\bm{b_2}}t^{(k_j-k_1)\bm{b_1}}t^{-k_j\bm{b_1}}v_0.\end{eqnarray*}
Since $h(-k_1\bm{b_1}+\bm{b_2})\neq0$, this implies $a_1=0$. Similarly,
we can prove $a_2=a_3=\cdots=a_n=0$. Therefore the conclusion follows.
\end{proof}

From the Lemmas \ref{izmoL} and \ref{nzHCmoL}, we have:

\begin{prop}\label{211}
For any $\Z$-basis $\{\bm{b_{1}}, \bm{b_{2}}\}$ of $\Gamma$.
Let $V$ be the Harish-Chandra module of $\wtL$. If there exist $\lambda_{1}, \lambda_{2}\in \C$ such that
$V_{(\lambda_{1}, \lambda_{2})}\neq0$ and $\mathcal{P}(V)\cap \big((\lambda_{1}, \lambda_{2})+\Z\bm{b_{1}}+
\N\bm{b_{2}}\big)=\emptyset$, then
$V\simeq M\big(\bm{b_{1}}, \bm{b_{2}}, T_{\rho}(\mathcal{H}_{\bm{b_{1}}})(\bm{c},\bm{\lambda})\big)$ for some $\rho\in\mathcal{E}_{\bm{b_1}}$.
\end{prop}

\begin{rmk}
If $V$ is a Harish-Chandra module $V$ of $\wtL$ satisfying the conditions in the Proposition $2.11$,
then $\bm{c}=(0,0,c_3,c_4)$, i.e., $f(\bm{b_1}),h(\bm{b_1})$ acting trivially.
\end{rmk}

As one of the main results in this paper, we prove that a Harish-Chandra module of $\wtL$ is
a generalized highest weight module or a uniformly bounded module. First, we need the following lemma.
\begin{lem}\label{213}
An irreducible weight $\wtL$-module $V$ is a generalized highest weight module if
there is a $\Z$-basis $\{\bm{b_{1}},\ \bm{b_{2}}\}$ of $\Gamma$ and a weight vector $v\neq0$
such that $E(\bm{b_{1}})v=E(\bm{b_{2}})v=t^{\bm{b_{1}}}v=0$.
\end{lem}

\begin{proof} Since there is a weight vector $v\neq0$, s.t., $E(\bm{b_{1}})v=E(\bm{b_{2}})v=t^{\bm{b_{1}}}v=0$, by induction, we have
\[E(\bm{m})v=t^{\bm{m}}v=0,\]for $\bm{m}\in\N\bm{b_{1}}+\N\bm{b_{2}}$ by induction.
Therefore, for $\bm{b_{1}}^{\prime}=2\bm{b_{1}}+\bm{b_{2}}$, $\bm{b_{2}}^{\prime}=3\bm{b_{1}}+\bm{b_{2}}\in\Gamma$,
we have \[E(\bm{m})v=t^{\bm{m}}v=0,\]for $\bm{m}\in\Z_{+}\bm{b_{1}}^{\prime}+\Z_{+}\bm{b_{2}}^{\prime}$.
It is obvious that $\{\bm{b_{1}}^{\prime},\ \bm{b_{2}}^{\prime}\}$ is a $\Z$-basis of $\Gamma$,
together with $V$ being irreducible, we get that $V$ is a generalized highest weight module.
\end{proof}

\begin{prop}
A Harish-Chandra module $V$ of $\wtL$ is a generalized highest weight module or a uniformly bounded module.
\end{prop}

\begin{proof} Let $(\lambda_{1},\ \lambda_{2})\in\mathcal{P}(V)$ and let
$V_{\bm{b}}:=V_{(\lambda_{1},\ \lambda_{2})+\bm{b}}$ for $\bm{b}\in\Gamma$.
Then $V=\oplus_{\bm{b}\in\Gamma}V_{\bm{b}}$. If $V$ is not a generalized highest weight module,
for $\bm{m}=(m_{1},m_{2})\in\Gamma$, consider the linear maps
$E(-m_{1}\bm{e_{1}}+\bm{e_{2}}): V_{(m_{1},m_{2})}\rightarrow V_{(0,m_{2}+1)}$,
$E((1-m_{1})\bm{e_{1}}+\bm{e_{2}}):V_{(m_{1},m_{2})}\rightarrow V_{(1,m_{2}+1)}$ and
$t^{-m_{1}\bm{e_{1}}+\bm{e_{2}}}: V_{(m_{1},m_{2})}\rightarrow V_{(0,m_{2}+1)}$. By the Lemma \ref{213}, we have that
\[\text{ker}\ E(-m_{1}\bm{e_{1}}+\bm{e_{2}})\cap\text{ker}\ E((1-m_{1})\bm{e_{1}}+\bm{e_{2}})
\cap\text{ker}\ t^{-m_{1}\bm{e_{1}}+\bm{e_{2}}}=0.\]
This shows that \[\text{dim}\ V_{(m_{1},m_{2})}\leq2\text{dim}\ V_{(0,m_{2}+1)}+\text{dim}\ V_{(1,m_{2}+1)}.\]
Now we consider the linear maps $E(-\bm{e_{1}}-m_{2}\bm{e_{2}}):V_{(0,m_{2}+1)}\rightarrow V_{(-1,1)}$,
$E(-\bm{e_{1}}+(1-m_{2})\bm{e_{2}}):V_{(0,m_{2}+1)}\rightarrow V_{(-1,2)}$ and
$t^{-\bm{e_{1}}-m_{2}\bm{e_{2}}}:V_{(0,m_{2}+1)}\rightarrow V_{(-1,1)}$. With the same reason, we get
\[\text{dim}\ V_{(0,m_{2}+1)}\leq2\text{dim}\ V_{(-1,1)}+\text{dim}\ V_{(-1,2)}.\]
Similarly, we have
\[\text{dim}\ V_{(1,m_{2}+1)}\leq2\text{dim}\ V_{(0,1)}+\text{dim}\ V_{(0,2)}.\]
Thus, $V$ is a uniformly bounded module.
\end{proof}

\section{Nonzero level Harish-Chandra Modules of $\wtL$}

In this section, we study the nonzero level Harish-Chandra module $V$ of $\wtL$, i.e., satisfying
$K_{i}.v=c_{i}v$ for $v\in V$, $\bm{0}\neq(c_{1}, c_{2}, c_{3}, c_{4})\in\C^4$.\\

We denote $[p, q]=\{x\mid x\in\Z, p\leq x\leq q\}$ and similarly for
$(-\infty, p]$, $[q, \infty)$ and $(-\infty, +\infty)$. First, we have:

\begin{thm}\label{nzlHCmiGHWm}
If $V$ is a nonzero level Harish-Chandra module of $\wtL$,
then $V$ is a GHW module.
\end{thm}

\begin{proof}
Without loss of generality, we may assume the center element $K_{1}$ acting
as $0\neq c_1\in\C$. Let $(\lambda_{1}, \lambda_{2})\in \mathcal{P}(V)$. Set $W_{0}:=\oplus_{i\in\Z}V_{(\lambda_1,\lambda_2)+i\bm{e_{1}}}\neq0$.
From the Theorem \ref{nzlham}, we see that $W_{0}$ as $\<E(k\bm{e_{1}}),t^{-k\bm{e_{1}}},K_{1}\mid k\in\N\>$-module is completely reducible.
Also from the Theorem \ref{nzlham}, we know that $V$ is not a uniformly bounded module. Thus $V$ is a GHW module.
\end{proof}

\begin{cor}
If $V$ is a uniformly bounded Harish-Chandra module of $\wtL$,
then $K_{i}.v=0$ for $v\in V$, $i=1,2,3,4$.
\end{cor}

We assume that $V=\oplus_{\bm{n}\in\Gamma}V_{\bm{\lambda}+\bm{n}}$ is a nontrivial GHW
Harish-Chandra $\wtL$-module with GHW $\bm{\lambda}=(\lambda_{1}, \lambda_{2})$ corresponding to
a $\Z$-basis $B=\{\bm{b_{1}}, \bm{b_{2}}\}$ of $\Gamma$. Without loss of generality, we assume
$\bm{\lambda}=\bm{0}$.

\begin{lem}\label{33}
(1). For any $v\in V$, there exists $p>0$ such that
$E(i\bm{b_{1}}+j\bm{b_{2}})v=t^{i\bm{b_{1}}+j\bm{b_{2}}}v=0$
for all $(i, j)\geq(p, p)$.\\
(2). For any $0\neq v\in V$, $(m_{1}, m_{2})>\bm{0}$, we have $E(-m_{1}\bm{b_{1}}-m_{2}\bm{b_{2}})v\neq0$.\\
(3). If $\bm{b}:=i_{1}\bm{b_{1}}+i_{2}\bm{b_{2}}\in\mathcal{P}(V)$,
then for any $(m_{1}, m_{2})>\bm{0}$, there exists $m\geq0$ such that
$\{x\in\Z\mid\bm{b}+x\bm{a}\in\mathcal{P}(V)\}=(-\infty, m]$,
where $\bm{a}=m_{1}\bm{b_{1}}+m_{2}\bm{b_{2}}$.
\end{lem}

\begin{proof}
Let $v_{0}$ be the GHW vector of $V$ corresponding to the $\Z$-basis $B$.\\
$(1)$ Since $v=uv_{0}$ for some $u\in\U(\wtL)$, then $u$ can be written as a linear combination
of elements of the form $u_{m, n}=t^{i_{1}\bm{b_{1}}+j_{1}\bm{b_{2}}}\cdots
t^{i_{m}\bm{b_{1}}+j_{m}\bm{b_{2}}}E(k_{1}\bm{b_{1}}+l_{1}\bm{b_{2}})\cdots
E(k_{n}\bm{b_{1}}+l_{n}\bm{b_{2}})$. Without loss of generality, we may assume
$u=u_{m, n}$. Take $p_{1}=-\Sigma_{i_{s}<0}i_{s}-\Sigma_{k_{t}<0}k_{t}+1$,
$p_{2}=-\Sigma_{j_{s}<0}j_{s}-\Sigma_{l_{t}<0}l_{t}+1$. Fix $m\in\Z_{+}$. By induction on $n$, one gets
$E(i\bm{b_{1}}+j\bm{b_{2}})v=t^{i\bm{b_{1}}+j\bm{b_{2}}}v=0$
for all $(i, j)\geq(p_{1}, p_{2})$, take $p=\text{max}\{p_{1}, p_{2}\}$, then the result follows.\\
$(2)$ Suppose $E(-m_{1}\bm{b_{1}}-m_{2}\bm{b_{2}})v=0$ for some $0\neq v\in V$ and some $(m_1,m_2)>\bm{0}$. Let $p$
be as in the proof $(1)$. Then one gets \[E(-m_{1}\bm{b_{1}}-m_{2}\bm{b_{2}})v
=E(\bm{b_{1}}+p(m_{1}\bm{b_{1}}+m_{2}\bm{b_{2}}))v=E(\bm{b_{2}}+p(m_{1}\bm{b_{1}}+m_{2}\bm{b_{2}}))v=0\]
\[t^{\bm{b_{1}}+p(m_{1}\bm{b_{1}}+m_{2}\bm{b_{2}})}v=t^{\bm{b_{2}}+p(m_{1}\bm{b_{1}}+m_{2}\bm{b_{2}})}v=0.\]
Note that the Lie algebra $L$ is generated by these elements, we have $Lv=0$,
which contradicts with $V$ is a nontrivial irreducible module.\\
$(3)$ See the Lemma $3.2$ in $\cite{LT1}$.
\end{proof}

The following lemma follows from the Lemma \ref{33} and the proof is given in \cite{LT1}.

\begin{lem}\label{34}
There exists a $\Z$-basis $B^{\prime}=\{\bm{b^{\prime}_{1}}, \bm{b^{\prime}_{2}}\}$
of $\Gamma$ such that\\
$(1)$ $V$ is a GHW module with weight $\bm{0}$ corresponding to the $\Z$-basis $B^{\prime}$;\\
$(2)$ $\{\Z_{+}\bm{b^{\prime}_{1}}+\Z_{+}\bm{b^{\prime}_{2}}\}\cap\mathcal{P}(V)={\bm{0}}$;\\
$(3)$ $\{-\Z_{+}\bm{b^{\prime}_{1}}-\Z_{+}\bm{b^{\prime}_{2}}\}\subseteq\mathcal{P}(V)$;\\
$(4)$ If $i_{1}\bm{b^{\prime}_{1}}+i_{2}\bm{b^{\prime}_{2}}\notin\mathcal{P}(V)$, then
$k_{1}\bm{b^{\prime}_{1}}+k_{2}\bm{b^{\prime}_{2}}\notin\mathcal{P}(V)$ for $(k_{1}, k_{2})\geq(i_{1}, i_{2})$;\\
$(5)$ If $i_{1}\bm{b^{\prime}_{1}}+i_{2}\bm{b^{\prime}_{2}}\in\mathcal{P}(V)$, then
$k_{1}\bm{b^{\prime}_{1}}+k_{2}\bm{b^{\prime}_{2}}\in\mathcal{P}(V)$ for $(k_{1}, k_{2})\leq(i_{1}, i_{2})$;\\
$(6)$ For any $\bm{0}\neq(k_{1}, k_{2})\geq\bm{0}$, $(i_{1}, i_{2})\in\Gamma$, we have
\[\{x\in\Z\mid i_{1}\bm{b^{\prime}_{1}}+i_{2}\bm{b^{\prime}_{2}}
+x(k_{1}\bm{b^{\prime}_{1}}+k_{2}\bm{b^{\prime}_{2}})\in\mathcal{P}(V)\}=(-\infty, m]\]
for some $m\in\Z$.
\end{lem}

From now on, we assume that $V$ is a nontrivial GHW Harish-Chandra module with GHW $\bm{0}$ corresponding
to the $\Z$-basis $B=\{\bm{b_{1}}, \bm{b_{2}}\}$ and $B$ satisfies the properties in the Lemma \ref{34}.
To characterize the nontrivial GHW Harish-Chandra module $V$ of $\wtL$, we need the
following lemmas. The proof is in \cite{LT1} (c.f. $\cite{LZ1}$, $\cite{Su1}$).

\begin{lem}\label{35}
If there exist an integer $s>0$ and $(i_{1}, i_{2})$, $(k_{1}, k_{2})\in\Gamma$
such that $k_{1}$, $k_{2}$ are coprime, and
\[\{i_{1}\bm{b_{1}}+i_{2}\bm{b_{2}}+x_{1}s\bm{b_{1}}+x_{2}s\bm{b_{2}}\mid(x_{1}, x_{2})\in\Gamma,
k_{1}x_{1}+k_{2}x_{2}=0\}\cap\mathcal{P}(V)=\emptyset,\]
then $V\simeq M\big(\bm{b^{\prime}_{1}}, \bm{b^{\prime}_{2}}, T_{\rho}(\mathcal{H}_{\bm{b^{\prime}_{1}}})(\bm{c},\bm{\lambda})\big)$ for
some $\Z$-basis $\{\bm{b^{\prime}_{1}}, \bm{b^{\prime}_{2}}\}$ of $\Gamma$ and some $\rho\in\mathcal{E}_{\bm{b^{\prime}_{1}}}$,
where $f(\bm{b^{\prime}_{1}}),h(\bm{b^{\prime}_{1}}),f(\bm{b^{\prime}_{2}}),h(\bm{b^{\prime}_{2}})$ act as scalars $c_1=0,c_2=0,c_3,c_4$
and $\bm{c}=(c_1,c_2,c_3,c_4)$.
\end{lem}

\begin{lem}\label{37}
If there exist $(i_{1}, i_{2})$, $(0, 0)\neq(k_{1}, k_{2})\in\Gamma$ such that
\[\{i_{1}\bm{b_{1}}+i_{2}\bm{b_{2}}+x(k_{1}\bm{b_{1}}+k_{2}\bm{b_{2}})\mid x\in\Z\}
\cap\mathcal{P}(V)=\emptyset,\]
then $V\simeq M\big(\bm{b^{\prime}_{1}}, \bm{b^{\prime}_{2}}, T_{\rho}(\mathcal{H}_{\bm{b^{\prime}_{1}}})(\bm{c},\bm{\lambda})\big)$ for
some $\Z$-basis $\{\bm{b^{\prime}_{1}}, \bm{b^{\prime}_{2}}\}$ of $\Gamma$ and some $\rho\in\mathcal{E}_{\bm{b^{\prime}_{1}}}$.
\end{lem}

\begin{lem}\label{38}
If there exist $(0, 0)\neq(m, n)\in\Gamma$, $(i, j)\in\Gamma$, $p$, $q\in\Z$
such that
\[\{x\in\Z\mid i\bm{b_{1}}+j\bm{b_{2}}+x(m\bm{b_{1}}+n\bm{b_{2}})\in\mathcal{P}(V)\}
\supseteq(-\infty, p]\cup[q, \infty),\]
then $V\simeq M\big(\bm{b^{\prime}_{1}}, \bm{b^{\prime}_{2}}, T_{\rho}(\mathcal{H}_{\bm{b^{\prime}_{1}}})(\bm{c},\bm{\lambda})\big)$ for
some $\Z$-basis $\{\bm{b^{\prime}_{1}}, \bm{b^{\prime}_{2}}\}$ of $\Gamma$ and some $\rho\in\mathcal{E}_{\bm{b^{\prime}_{1}}}$.
\end{lem}

\begin{lem}\label{39}
If there exist $(i, j)$, $(k, l)\in\Gamma$ and $x_{1}, x_{2}, x_{3}\in\Z$
with $x_{1}<x_{2}<x_{3}$ such that
\begin{align}i\bm{b_{1}}+j\bm{b_{2}}+x_{1}(k\bm{b_{1}}+l\bm{b_{2}})\notin\mathcal{P}(V),\end{align}
\begin{align}i\bm{b_{1}}+j\bm{b_{2}}+x_{2}(k\bm{b_{1}}+l\bm{b_{2}})\in\mathcal{P}(V),\end{align} and
\begin{align}i\bm{b_{1}}+j\bm{b_{2}}+x_{3}(k\bm{b_{1}}+l\bm{b_{2}})\notin\mathcal{P}(V),\end{align}
then $V\simeq M\big(\bm{b^{\prime}_{1}}, \bm{b^{\prime}_{2}}, T_{\rho}(\mathcal{H}_{\bm{b^{\prime}_{1}}})(\bm{c},\bm{\lambda})\big)$ for
some $\Z$-basis $\{\bm{b^{\prime}_{1}}, \bm{b^{\prime}_{2}}\}$ of $\Gamma$ and some $\rho\in\mathcal{E}_{\bm{b^{\prime}_{1}}}$.
\end{lem}

\begin{proof}
Without loss of generality, we may assume $k, l$ are coprime. Thus we can choose
$(m, n)\in\Gamma$ with $kn-lm=1$. Let $\bm{b^{\prime}_{1}}=k\bm{b_{1}}+l\bm{b_{2}}$,
$\bm{b^{\prime}_{2}}=m\bm{b_{1}}+n\bm{b_{2}}$, then
$\{\bm{b^{\prime}_{1}}, \bm{b^{\prime}_{2}}\}$ is a $\Z$-basis of $\Gamma$.
Replacing $x_{2}$ by the largest $x<x_{3}$ with
$i\bm{b_{1}}+j\bm{b_{2}}+x(k\bm{b_{1}}+l\bm{b_{2}})\in\mathcal{P}(V)$,
then replacing $x_{3}$ by $x_{2}+1$ and $(i, j)$ by $(i, j)+x_{2}(k, l)$, we can assume
\begin{align}x_{1}<x_{2}=0<x_{3}=1.\end{align}
We may assume that there exists $s\in\Z$
with \begin{align}i\bm{b_{1}}+j\bm{b_{2}}+\bm{b^{\prime}_{2}}+s\bm{b^{\prime}_{1}}
=(i+m)\bm{b_1}+(j+n)\bm{b_2}+s(k\bm{b_{1}}+l\bm{b_{2}})\notin\mathcal{P}(V).\end{align}
Otherwise from the Lemma \ref{38}, we are done.
Thus by $(3.1)$, $(3.2)$, $(3.3)$, $(3.4)$ and $(3.5)$, we have
\[E(x_{1}\bm{b^{\prime}_{1}})v_{i\bm{b_{1}}+j\bm{b_{2}}}=E(x_1(k\bm{b_{1}}+l\bm{b_{2}}))v_{i\bm{b_{1}}+j\bm{b_{2}}}=0,\ t^{x_{1}\bm{b^{\prime}_{1}}}v_{i\bm{b_{1}}+j\bm{b_{2}}}=t^{x_1(k\bm{b_{1}}+l\bm{b_{2}})}v_{i\bm{b_{1}}+j\bm{b_{2}}}=0,\]
\[E(\bm{b^{\prime}_{1}})v_{i\bm{b_{1}}+j\bm{b_{2}}}=E(k\bm{b_{1}}+l\bm{b_{2}})v_{i\bm{b_{1}}+j\bm{b_{2}}}=0,\ t^{\bm{b^{\prime}_{1}}}v_{i\bm{b_{1}}+j\bm{b_{2}}}=t^{k\bm{b_{1}}+l\bm{b_{2}}}v_{i\bm{b_{1}}+j\bm{b_{2}}}=0,\]and
\[E(\bm{b^{\prime}_{2}}+s\bm{b^{\prime}_{1}})v_{i\bm{b_{1}}+j\bm{b_{2}}}=0,\
t^{\bm{b^{\prime}_{2}}+s\bm{b^{\prime}_{1}}}v_{i\bm{b_{1}}+j\bm{b_{2}}}=0,\]
where $0\neq v_{i\bm{b_{1}}+j\bm{b_{2}}}\in V_{i\bm{b_{1}}+j\bm{b_{2}}}$.
Note that since $x_{1}<0$, we have that $\{E(p\bm{b^{\prime}_{1}}+q\bm{b^{\prime}_{2}}),\ t^{p\bm{b^{\prime}_{1}}+q\bm{b^{\prime}_{2}}}\mid
p\in\Z, q\in\N\}$ belongs to the subalgebra generated by
\[\{E(x_{1}\bm{b^{\prime}_{1}}),\ E(\bm{b^{\prime}_{1}}),\ E(\bm{b^{\prime}_{2}}+s\bm{b^{\prime}_{1}}),\
t^{x_{1}\bm{b^{\prime}_{1}}},\ t^{\bm{b^{\prime}_{1}}},\ t^{\bm{b^{\prime}_{2}}+s\bm{b^{\prime}_{1}}}\}.\] We obtain $E(p\bm{b^{\prime}_{1}}+q\bm{b^{\prime}_{2}})v_{i\bm{b_{1}}+j\bm{b_{2}}}
=t^{p\bm{b^{\prime}_{1}}+q\bm{b^{\prime}_{2}}}v_{i\bm{b_{1}}+j\bm{b_{2}}}=0$ for
$p\in\Z, q\in\N$.
Since $\{\bm{b^{\prime}_{1}}, \bm{b^{\prime}_{2}}\}$ is a $\Z$-basis of $\Gamma$, and $V$ is irreducible, from PBW theorem, we have
$V=\U(\wtL)v_{i\bm{b_{1}}+j\bm{b_{2}}}$ and
\[\{i\bm{b_{1}}+j\bm{b_{2}}+\Z\bm{b^{\prime}_{1}}+\N\bm{b^{\prime}_{2}}\}\cap\mathcal{P}(V)
=\emptyset.\] Thus the results follows from the Proposition \ref{211}.
\end{proof}

\begin{lem}\label{310}
If there exist $i>0$, $j<0$ and $0\neq v_{\bm{a}}\in V_{\bm{a}}$,
$\bm{b}=m\bm{b_{1}}+n\bm{b_{2}}\neq\bm{0}$, such that $E(i\bm{b})v_{\bm{a}}=0$,
$E(j\bm{b})v_{\bm{a}}=0$,
then $V\simeq M\big(\bm{b^{\prime}_{1}}, \bm{b^{\prime}_{2}}, T_{\rho}(\mathcal{H}_{\bm{b^{\prime}_{1}}})(\bm{c},\bm{\lambda})\big)$ for
some $\Z$-basis $\{\bm{b^{\prime}_{1}}, \bm{b^{\prime}_{2}}\}$ of $\Gamma$ and some $\rho\in\mathcal{E}_{\bm{b^{\prime}_{1}}}$.
\end{lem}

\begin{proof}
Write $(m, n)=s(m^{\prime}, n^{\prime})$ with
$m^{\prime}, n^{\prime}$ coprime and $s\geq1$. Then we can choose $(m_{2}, n_{2})\in\Gamma$ with
$n^{\prime}m_{2}-m^{\prime}n_{2}=1$. Let $\bm{b^{\prime}_{1}}=m^{\prime}\bm{b_{1}}+n^{\prime}\bm{b_{2}}$,
$\bm{b^{\prime}_{2}}=m_{2}\bm{b_{1}}+n_{2}\bm{b_{2}}$, then
$\{\bm{b^{\prime}_{1}}, \bm{b^{\prime}_{2}}\}$ is a $\Z$-basis of $\Gamma$. Fix any $0\neq q\in\Z$.\\
$\textbf{Case\ 1}.$ If $\{\bm{a}+q\bm{b^{\prime}_{2}}+x\bm{b^{\prime}_{1}}\mid x\in\Z\}\cap\mathcal{P}(V)=\emptyset$,
then by the Lemma $3.7$, we are done.\\
$\textbf{Case\ 2}.$ If there exist integers $x_1<x_2<x_3$ with $\bm{a}+q\bm{b^{\prime}_{2}}+x_2\bm{b^{\prime}_{1}}\in\mathcal{P}(V)$ and
$\bm{a}+q\bm{b^{\prime}_{2}}+x_i\bm{b^{\prime}_{1}}\notin\mathcal{P}(V)$, $i=1,3$, then by the Lemma \ref{310}, we are done.\\
$\textbf{Case\ 3}.$ If there exist $m,n\in\Z$ with $(-\infty,m]\cup[n,\infty)\subseteq\{x\in\Z\mid\bm{a}+q\bm{b^{\prime}_{2}}+x\bm{b^{\prime}_{1}}\in\mathcal{P}(V)\}$,
then by the Lemma \ref{39}, we are done.

Now if the above three cases don't occur, we know that there exists some integer $p_{q}$ such that
$A_{q}:=\{x\in\Z\mid\bm{a}+q\bm{b^{\prime}_{2}}+x\bm{b^{\prime}_{1}}\in\mathcal{P}(V)\}
=(-\infty, p_{q}]$ or $[p_{q}, \infty)$. We first assume $A_{q}=(-\infty, p_{q}]$. Thus
\[E(q\bm{b^{\prime}_{2}}-jxs\bm{b^{\prime}_{1}}\pm\bm{b^{\prime}_{1}})v_{\bm{a}}
=t^{q\bm{b^{\prime}_{2}}-jxs\bm{b^{\prime}_{1}}\pm\bm{b^{\prime}_{1}}}v_{\bm{a}}=0\]
for sufficient large integer $x>0$. Since $E(j\bm{b})v_{\bm{c}}=E(js\bm{b^{\prime}_{1}})v_{\bm{a}}=0$, we can obtain
\[E(q\bm{b^{\prime}_{2}}\pm\bm{b^{\prime}_{1}})v_{\bm{a}}
=t^{q\bm{b^{\prime}_{2}}\pm\bm{b^{\prime}_{1}}}v_{\bm{a}}=0.\] If $A_{q}=[p_{q}, \infty)$, with similar argument,
we can also obtain \[E(q\bm{b^{\prime}_{2}}\pm\bm{b^{\prime}_{1}})v_{\bm{a}}
=t^{q\bm{b^{\prime}_{2}}\pm\bm{b^{\prime}_{1}}}v_{\bm{a}}=0.\]
This implies
\[E(\pm(\bm{b^{\prime}_{1}}+\bm{b^{\prime}_{2}}))v_{\bm{a}}
=E(\pm(\bm{b^{\prime}_{1}}+2\bm{b^{\prime}_{2}}))v_{\bm{a}}=0,\ t^{\pm(\bm{b^{\prime}_{1}}+\bm{b^{\prime}_{2}})}v_{\bm{a}}
=t^{\pm(\bm{b^{\prime}_{1}}+2\bm{b^{\prime}_{2}})}v_{\bm{a}}=0.\]
Since $\{\bm{b^{\prime}_{1}}+\bm{b^{\prime}_{2}},\bm{b^{\prime}_{1}}+2\bm{b^{\prime}_{2}}\}$ is a $\Z$-basis of $\Gamma$, $L$ is generated by $\{E(\pm(\bm{b^{\prime}_{1}}+\bm{b^{\prime}_{2}})),
E(\pm(\bm{b^{\prime}_{1}}+2\bm{b^{\prime}_{2}})), t^{\pm(\bm{b^{\prime}_{1}}+\bm{b^{\prime}_{2}})},
t^{\pm(\bm{b^{\prime}_{1}}+2\bm{b^{\prime}_{2}})}\}$.
Thus $V=\U(\wtL)v_{\bm{a}}$ is a trivial module, which is a contradiction.
\end{proof}

The following proposition give the characterization of the nontrivial GHW Harish-Chandra module.

\begin{prop}\label{311}
If $V$ is a nontrivial GHW Harish-Chandra module with GHW $\bm{\lambda}=(\lambda_{1}, \lambda_{2})$
corresponding to a $\Z$-basis $B=\{\bm{b_{1}}, \bm{b_{2}}\}$ of $\Gamma$ of $\wtL$,
then $V\simeq M\big(\bm{b^{\prime}_{1}}, \bm{b^{\prime}_{2}}, T_{\rho}(\mathcal{H}_{\bm{b^{\prime}_{1}}})(\bm{c},\bm{\lambda})\big)$ for
some $\Z$-basis $\{\bm{b^{\prime}_{1}}, \bm{b^{\prime}_{2}}\}$ of $\Gamma$ and some $\rho\in\mathcal{E}_{\bm{b^{\prime}_{1}}}$.
\end{prop}

\begin{proof}
From the Lemma \ref{310} and the proof of the Proposition $3.9$ in \cite{LT1}, we can obtain our results.
\end{proof}

Together with the Theorem \ref{nzlHCmiGHWm} and the Proposition \ref{311}, we have:

\begin{thm}
If $V$ is a nonzero level Harish-Chandra $\wtL$-module,
then \[V\simeq M\big(\bm{b^{\prime}_{1}}, \bm{b^{\prime}_{2}}, T_{\rho}(\mathcal{H}_{\bm{b^{\prime}_{1}}})(\bm{c},\bm{\lambda})\big)\]
for some $\Z$-basis $\{\bm{b^{\prime}_{1}}, \bm{b^{\prime}_{2}}\}$ of $\Gamma$ and some
$\rho\in\mathcal{E}_{\bm{b^{\prime}_{1}}}$, $\bm{\lambda}\in\C^2$.
\end{thm}

\section{Classification of GHW Harish-Chandra $\wtL$-modules}
In this section, we give the classification of GHW Harish-Chandra modules of $\wtL$ by
using the highest weight modules of $L$. From the Proposition \ref{311},
we only need to find in which case that the irreducible GHW $\wtL$-module
$M\big(\bm{b_{1}}, \bm{b_{2}}, T_{\rho}(\mathcal{H}_{\bm{b_{1}}})(\bm{c},\bm{\lambda})\big)$ is a Harish-Chandra module.

First we give a triangular decomposition of $L$ and construct a
class of $\Z$-graded irreducible highest weight modules of $L$. Recall that
\[\wtL_{i}=\<E(m\bm{b_{1}}+i\bm{b_{2}}),\ t^{m\bm{b_{1}}+i\bm{b_{2}}}\mid m\in\Z\>,\  i\in\Z\setminus\{0\},\]
and \[\wtL_+=\bigoplus_{i>0}\wtL_{i},\ \wtL_-=\bigoplus_{i<0}\wtL_{i}.\]
Then $L=\wtL_+\oplus \mathcal{H}_{\bm{b_{1}}}\oplus \wtL_-$.

\begin{rmk}\label{41}
In this section, we call a $L$-module $V$ the highest weight module (corresponding to the $\Z$-basis $\{\bm{b_1},\bm{b_2}\}$)
if there exists a nonzero $v\in V$ such that $V=\U(L)v$ and $\wtL_+.v=0$.
\end{rmk}

 For any linear function
$\rho: \mathcal{H}_{\bm{b_{1}}}\rightarrow\C$ with $\rho(f(\bm{b_1}))=\rho(h(\bm{b_1}))=0$, we define
a one dimensional $(\mathcal{H}_{\bm{b_{1}}}\oplus \wtL_+)$-module $\C v_0$ as follows:
\begin{align}\wtL_+.v_0=0,\ x.v_0=\rho(x)v_0,\ x\in \mathcal{H}_{\bm{b_{1}}}.\end{align}
Then we have an induced $L$-module \begin{align}\overline{V}(\rho)=\text{Ind}_{\mathcal{H}_{\bm{b_{1}}}\oplus \wtL_+}^{L}\C v_0
=\U(L)\otimes_{\U(\mathcal{H}_{\bm{b_{1}}}\oplus \wtL_+)}\C v_0.\end{align}
We see that $\overline{V}(\rho)$ is a $\Z$-graded module. It is clear that $\overline{V}(\rho)$ has
a unique maximal $\Z$-graded submodule $J(\rho)$. Then we obtain a $\Z$-graded irreducible highest
weight $L$-module \[V(\rho)=\overline{V}(\rho)/J(\rho)=\oplus_{i\in\Z}V(\rho)_i,\]
where \[\begin{aligned}V(\rho)_i=\text{Span}_{\C}\{E(i_1\bm{b_1}+j_1\bm{b_2})E(i_2\bm{b_1}+j_2\bm{b_2})\cdots
E(i_m\bm{b_1}+j_m\bm{b_2})t^{s_1\bm{b_1}+k_1\bm{b_2}}\\
\cdots t^{s_n\bm{b_1}+k_n\bm{b_2}}v_0\mid m,n\in\Z_+,\ \sum_{p=1}^{m}j_p+\sum_{p=1}^{n}k_p=i\},\ \text{for}\ i\in\Z.
\end{aligned}\]
We call $V(\rho)_i$ for $i\in\Z$ the weight space of the $L$-module $V(\rho)$. If dim $V(\rho)_i<\infty$, we say the weight space
$V(\rho)_i$ is finite dimensional.

For the later use, we need a conception of
exp-polynomial function. Recall from \cite{BZ} that a function $f:\Z\rightarrow \C$ is said to be
exp-polynomial if it can be written as a finite sum
\[f(n)=\sum c_{m,a}n^{m}a^{n},\]
for some $c_{m,a}\in \C$, $m\in\Z_+$ and $0\neq a\in\C$.

The following lemma is due to \cite{W}.

\begin{lem}\label{43}
A function $f:\Z\rightarrow\C$ is an {\em exp-polynomial function} if and only
if there exist $a_0,...,a_n\in\C$ with $a_0a_n\neq0$, such that
\[\sum_{i=0}^na_if(m+i)=0,\]for all $m\in\Z$.
\end{lem}

\begin{rmk}
In general, for fixed $a_0,...,a_n\in\C$ with $a_0a_n\neq0$, the exp-polynomial function $f$
satisfying $\sum_{i=0}^na_if(m+i)=0$, $\forall m\in\Z$, is not unique.
\end{rmk}

Then we have the following result.

\begin{prop}\label{fws}
Suppose the linear function $\rho: \mathcal{H}_{\bm{b_{1}}}\rightarrow\C$ such that $\rho(f(\bm{b_1}))=\rho(h(\bm{b_1}))=0$.
Then the $\Z$-graded $L$-module $V(\rho)$ has finite dimensional weight spaces if and
only if there exist two exp-polynomials $g_j:\Z\rightarrow\C$ satisfying
$\sum_{i=0}^na_ig_j(k+i)=0$ for $j=1,2$, $k\in\Z$, $a_i\in\C$, $a_0a_n\neq0$ and
\[g_1(0)=\text{det}{\bm{b_1} \choose \bm{b_2}}\rho(f(\bm{b_2})),\
g_2(0)=\text{det}{\bm{b_1} \choose \bm{b_2}}\rho(h(\bm{b_2})),\]
\[g_1(m)=\rho(mE(m\bm{b_1})),\ g_2(m)=\rho(mt^{m\bm{b_1}}),\ m\in\Z\setminus\{0\}.\]
\end{prop}

\begin{proof}
First, we define two linear maps $\phi_1,\phi_2:\C[t_1^{\pm1},t_2^{\pm1}]\rightarrow L$ by
\[\phi_i(t_1^{m_1}t_2^{m_2})=\begin{cases}E(m_1\bm{b_{1}}+m_2\bm{b_{2}}),\ &i=1,\cr
t^{m_1\bm{b_{1}}+m_2\bm{b_{2}}},\ &i=2.\end{cases}\]
If $V(\rho)$ has finite dimensional weight spaces, since dim $V(\rho)_{-1}<\infty$ and $\phi_1(t_1^it_2^{-1})v_0\in V(\rho)_{-1}$
for all $i\in\Z$, there exists $k\in\Z$ and nonzero polynomials
$P(t_1)=\sum_{i=0}^{n}a_{i}t_1^i\in\C[t_1]$ with $a_{0}a_{n}\neq0$ such that
\[\phi_1(t_2^{-1}t_1^{k}P(t_1))v_0=0.\]
Applying $\phi_i(t_1^st_2)$ for any $s\in\Z$, $i=1,2$ to the above equation respectively, we get that
\begin{align}\big(\sum_{i=0}^{n}a_i(k+s+i)E((k+s+i)\bm{b_1})+\text{det}{\bm{b_1} \choose \bm{b_2}}a_{-k-s}f(\bm{b_2})\big).v_0=0,\end{align}
and
\begin{align}\big(\sum_{i=0}^{n}a_i(k+s+i)t^{(k+s+i)\bm{b_1}}+\text{det}{\bm{b_1} \choose \bm{b_2}}a_{-k-s}h(\bm{b_2})\big).v_0=0,\end{align}
where $a_{-k-s}=0$ if $-k-s\notin\{0,1,...,n\}$. Set $g_1:\Z\rightarrow\C$ such that
$g_1(0)=\text{det}{\bm{b_1} \choose \bm{b_2}}\rho(f(\bm{b_2}))$ and $g_1(m)=\rho(mE(m\bm{b_1}))$ for $m\in\Z\setminus\{0\}$.
Then $(4.3)$ becomes \[\sum_{i=0}^na_ig_1(m+i)=0,\ \forall m\in\Z.\]
Set $g_2:\Z\rightarrow\C$ such that
$g_2(0)=\text{det}{\bm{b_1} \choose \bm{b_2}}\rho(h(\bm{b_2}))$ and $g_2(m)=\rho(mt^{m\bm{b_1}})$ for $m\in\Z\setminus\{0\}$.
Then $(4.4)$ becomes \[\sum_{i=0}^na_ig_2(m+i)=0,\ \forall m\in\Z.\]
From the Lemma \ref{43}, we have that $g_1$, $g_2$ are exp-polynomial functions.

\vspace{3mm}

Conversely, we use the Theorem $1.7$ from \cite{BZ} to prove that
the $\Z$-graded $L$-module $V(\rho)$ has finite dimensional weight spaces, i.e., dim $V(\rho)_i<\infty$ for all $i\in\Z$.
Since $\C v_0$ is a one dimensional $\mathcal{H}_{\bm{b_{1}}}$-module with exp-polynomial action, i.e.,
$\mathcal{H}_{\bm{b_{1}}}$ acts on $\C v_0$ through two exp-polynomials $g_1,g_2$, and $\wtL_+.v_0=0$,
from the Theorem 1.7 in \cite{BZ}, we just need to prove that $L$ is $\Z$-extragraded.
Set the index sets $X_i=\{(1,i),(2,i)\}$ for $i\in\Z\setminus\{0\}$
and $X_0=\{(i,0)\mid i=1,2,...,6\}$. For $i\in\Z\setminus\{0\},$ let \[\mathcal{L}_{\bm{k}}^{i}(j)=\begin{cases}
E(i\mathbf{b_{2}}+j\mathbf{b_{1}}),\ &\bm{k}=(1,i),\cr t^{i\mathbf{b_{2}}+j\mathbf{b_{1}}},\ &\bm{k}=(2,i),\ j\in\Z,
\end{cases}\]
and \[\mathcal{L}_{\bm{k}}^{0}(j)=\begin{cases}
jE(j\mathbf{b_{1}}),\ &\bm{k}=(1,0),\ j\neq0,\cr jt^{j\mathbf{b_{1}}},\ &\bm{k}=(2,0),\ j\neq0,
\cr K_i,\ &\bm{k}=(i+2,0),\ i=1,2,3,4,\ j=0.\end{cases}\]
$\textbf{Claim\ 1}.$ $L$ is a $\Z$-graded exp-polynomial Lie algebra (see the Definition $1.2$ in \cite{BZ}).

In fact, let $L=\oplus_{j\in\Z}L(j)$, where $L(j)=\<\mathcal{L}_{\bm{k}}^{i}(j)\mid i\in\Z,\ \bm{k}\in X_i\>$.
$[L(j_1),L(j_2)]\subseteq L(j_1+j_2)$ for $j_1,j_2\in\Z$. Thus $L$ is $\Z$-graded. And it is straightforward to check that
$L$ is an exp-polynomial Lie algebra with the distinguished spanning set $\{\mathcal{L}_{\bm{k}}^{i}(j)\mid\bm{k}\in X_i,\ i,j\in\Z\}$.\\
$\textbf{Claim\ 2}.$ The $\Z$-graded exp-polynomial Lie algebra $L$ is $\Z$-extragraded (see the Definition $1.4$ in \cite{BZ}).

In fact, let $L=\oplus_{i\in\Z}L^{(i)}$, where $L^{(i)}=\<\mathcal{L}_{\bm{k}}^{i}(j)\mid j\in\Z,\ \bm{k}\in X_i\>$.
$[L^{(i_1)},L^{(i_2)}]\subseteq L^{(i_1+i_2)}$ for $i_1,i_2\in\Z$, i.e., $L$ has another $\Z$-gradation.
\end{proof}

For linear function $\rho: \mathcal{H}_{\bm{b_{1}}}\rightarrow\C$ with $\rho(f(\bm{b_1}))=\rho(h(\bm{b_1}))=0$,
we say that $\rho$ is an {\em exp-polynomial function over $\mathcal{H}_{\bm{b_{1}}}$} if
there exist $a_0,...,a_n\in\C$, $a_0a_n\neq0$ and two exp-polynomials $g_0,g_1$ given by
\[g_1(0)=\text{det}{\bm{b_1} \choose \bm{b_2}}\rho(f(\bm{b_2})),\ g_2(0)=\text{det}{\bm{b_1} \choose \bm{b_2}}\rho(h(\bm{b_2})),\]
and
\[g_1(m)=\rho(mE(m\bm{b_1})),\ g_2(m)=\rho(mt^{m\bm{b_1}})\]for all $m\in\Z\setminus\{0\}$
such that $\sum_{i=0}^na_ig_j(k+i)=0$ for $j=1,2$, $k\in\Z$.

Let
\[
{\bm{b_{1}} \choose \bm{b_{2}}}^{-1}=\left(\begin{array}{cc}
p_1 & q_1\\
p_2 & q_2
\end{array}\right)\in\text{GL}_{2\times2}(\Z).
\]Set $\widetilde{d_1}=p_1d_1+p_2d_2$, $\widetilde{d_2}=q_1d_1+q_2d_2$. Then we have
\[[\widetilde{d_i},E(m_1\bm{b_1}+m_2\bm{b_2})]=m_iE(m\bm{b_1}+n\bm{b_2}),\
[\widetilde{d_i},t^{m_1\bm{b_1}+m_2\bm{b_2}}]=m_it^{m_1\bm{b_1}+m_2\bm{b_2}}\]
for $i=1,2$, $m_1,m_2\in\Z$.

Now we construct a class of $\Z^2$-graded irreducible GHW $\wtL$-modules by using the above $\Z$-graded
highest weight $L$-module $V(\rho)$. For any linear function
$\rho: \mathcal{H}_{\bm{b_{1}}}\rightarrow\C$ with $\rho(f(\bm{b_1}))=\rho(h(\bm{b_1}))$=0, we set
$\widehat{V}(\rho)=V(\rho)\otimes\C[t^{\pm1}]$, and define the actions of $\wtL$ on
$\widehat{V}(\rho)$ as follows:
\[E(m\bm{b_1}+n\bm{b_2}).(v\otimes t^k)=(E(m\bm{b_1}+n\bm{b_2}).v)\otimes t^{m+k},\
t^{m\bm{b_1}+n\bm{b_2}}.(v\otimes t^k)=(t^{m\bm{b_1}+n\bm{b_2}}.v)\otimes t^{m+k}\]
\[\widetilde{d_1}.(v\otimes t^k)=k(v\otimes t^k),\ \widetilde{d_2}.(v\otimes t^k)=j(v\otimes t^k),\]
\[K_i.(v\otimes t^k)=(K_i.v)\otimes t^k\]
for $(m,n)\in\Z^2\setminus\{\bm{0}\}$, $v\in V(\rho)_j$, $j\in\Z$, $i=1,2,3,4$.
It is clear that $\widehat{V}(\rho)$ is a $\Z^2$-graded $\wtL$-module. And
\[\widehat{V}(\rho)=\bigoplus_{m,n\in\Z}\widehat{V}(\rho)_{(m,n)},\]
where $\widehat{V}(\rho)_{(m,n)}=V(\rho)_m\otimes t^n$. We call $\widehat{V}(\rho)_{(m,n)}$, $m,n\in\Z$
weight spaces of the module $\widehat{V}(\rho)$
with respect to $\widetilde{d_1}$, $\widetilde{d_2}$.

Let $W(i)$ be the $\wtL$-submodule of $\widehat{V}(\rho)$ generated by $v_0\otimes t^i$, $i\in\Z$, where $v_0$ is defined in $(4.1)$.

\begin{lem}\label{46}
Let $\rho\in\mathcal{E}_{\bm{b_{1}}}$.\\
(1). If $T_{\rho}(\mathcal{H}_{\bm{b_{1}}})=T_0$, $\widehat{V}(\rho)=\bigoplus_{i\in\Z}W(i)$.\\
(2). If $T_{\rho}(\mathcal{H}_{\bm{b_{1}}})=T_r$ for some $r\in\N$,
$\widehat{V}(\rho)=\bigoplus_{i=0}^{r-1}W(i)$.\\
Where $W(i)$ is a $\Z^2$-graded irreducible $\wtL$-submodule of $\widehat{V}(\rho)$.
\end{lem}

\begin{proof}
We need to notice the following two facts. First, any nonzero $\wtL$-submodule of $\widehat{V}(\rho)$ contains $v_0\otimes t^i$ for some $i\in\Z$.
Second, the two $\wtL$-submodules $W(m)=W(n)$ if and only if $t^{m-n}\in T_r$, where $T_r=T_{\rho}(\mathcal{H}_{\bm{b_{1}}})$, $r\in\Z_+$.
For $(1)$, $W(i)$ is an $\Z^2$-graded irreducible $\wtL$-module
follows from that $V(\rho)$ is an irreducible $L$-module. For $(2)$, let $M$ be a nonzero submodule of the $\wtL$-module $W(i)$, then
$v_0\otimes t^n\in M$ for some $n\in\Z$. Since $\U(\mathcal{H}_{\bm{b_{1}}})(v_0\otimes t^i)=v_0\otimes(T_r\cdot t^i)$ and $v_0\otimes t^n\in\U(\mathcal{H}_{\bm{b_{1}}})(v_0\otimes t^i)$, we have $t^{n}\in T_r\cdot t^i$.
This implies that $v_0\otimes t^i\in W(n)\subseteq M$, i.e., $W(i)\subseteq M$. Thus $M=W(i)$, which shows that $W(i)$ is irreducible.
\end{proof}

For $\rho\in\mathcal{E}_{\bm{b_{1}}}$, we know that there exists a unique maximal $\Z^2$-graded submodule
$J(i)$ of $\widehat{V}(\rho)$ which insects $W(i)$ trivially by the Lemma \ref{46}.
Then we get the $\Z^2$-graded irreducible $\wtL$-module
\[\widehat{V}(\rho,i)=\widehat{V}(\rho)/J(i)\simeq W(i).\]

\begin{rmk}(1). From the Lemma $3.3$ in $\cite{W}$, we see $\rho\in\mathcal{E}_{\bm{b_{1}}}$ if $\rho$ is an exp-polynomial
function over $\mathcal{H}_{\bm{b_{1}}}$.\\
(2). For $\rho\in\mathcal{E}_{\bm{b_{1}}}$, we have $W(i)\simeq W(j)$, $i,j\in\Z$ from $(2.4)$ and the Lemma \ref{46}.
\end{rmk}

\begin{lem}\label{47}
(1). For any linear function
$\rho: \mathcal{H}_{\bm{b_{1}}}\rightarrow\C$ with $\rho(f(\bm{b_1}))=\rho(h(\bm{b_1}))=0$,
the $\wtL$-module $\widehat{V}(\rho)$ has finite dimensional weight spaces
if and only if $L$-module $V(\rho)$ has finite weight spaces.\\
(2). For $\rho\in\mathcal{E}_{\bm{b_1}}$, $M\big(\bm{b_{1}}, \bm{b_{2}}, T_{\rho}(\mathcal{H}_{\bm{b_{1}}})(\bm{c},\bm{\lambda})\big)\simeq
\widehat{V}(\rho,0)$ as $\wtL$-module.
\end{lem}

\begin{proof}
(1). Since $\widehat{V}(\rho)_{(m,n)}=V(\rho)_m\otimes t^n$, $m,n\in\Z$, the first assertion is obvious.\\
(2). Note that $\U(\mathcal{H}_{\bm{b_{1}}}).(v_0\otimes 1)\simeq T_{\rho}(\mathcal{H}_{\bm{b_{1}}})$ for $\rho\in\mathcal{E}_{\bm{b_1}}$
and $\wtL_{+}.\big(\U(\mathcal{H}_{\bm{b_{1}}}).(v_0\otimes 1)\big)=0$. Then the result follows the irreducibility of $\widehat{V}(\rho,0)$.
\end{proof}

By the Lemma \ref{47}, together with the Proposition \ref{fws}, we obtain the main result in this section.
\begin{thm}
For $\rho\in\mathcal{E}_{\bm{b_1}}$, the irreducible GHW $\wtL$-module
$M\big(\bm{b_{1}}, \bm{b_{2}}, T_{\rho}(\mathcal{H}_{\bm{b_{1}}})(\bm{c},\bm{\lambda})\big)$
is a Harish-Chandra module if and only if $\rho$ is an exp-polynomial
function over $\mathcal{H}_{\bm{b_{1}}}$.
\end{thm}

\begin{rmk}
(1). For $\rho\in\mathcal{E}_{\bm{b_1}}$,
the irreducible GHW $\wtL$-module $M\big(\bm{b_{1}}, \bm{b_{2}}, T_{\rho}(\mathcal{H}_{\bm{b_{1}}})(\bm{c},\bm{\lambda})\big)$
is a one dimensional trivial module if and only if $\bm{c}=\bm{0}\in\C^4$ and $T_{\rho}(\mathcal{H}_{\bm{b_{1}}})=T_0$.\\
(2). If $\rho$ is an exp-polynomial
function over $\mathcal{H}_{\bm{b_{1}}}$, then we have that the GHW Harish-Chandra
$M\big(\bm{b_{1}}, \bm{b_{2}}, T_{\rho}(\mathcal{H}_{\bm{b_{1}}})(\bm{c},\bm{\lambda})\big)$
is a one dimensional trivial module if and only if $\rho=0$, i.e., $T_{\rho}(\mathcal{H}_{\bm{b_{1}}})=T_0$.
\end{rmk}

\end{document}